\documentclass[11pt]{article}
\usepackage{fix-cm}
\usepackage{geometry}
\usepackage{amsmath}
\usepackage{amssymb}
\usepackage{amsthm}
\usepackage{graphicx}
\usepackage{multicol}
\usepackage{microtype}
\usepackage{url}
\usepackage[authoryear]{natbib}
\usepackage{float}

\usepackage{booktabs}
\usepackage{array}

\usepackage{algorithm}
\usepackage{algpseudocode}
\algrenewcommand\algorithmiccomment[1]{\hfill{\footnotesize // #1}}

\usepackage{etoolbox}
\AtBeginEnvironment{proof}{\par\smallskip}
\AtEndEnvironment{proof}{\par\smallskip}
\AtBeginEnvironment{proposition}{\par\smallskip}
\AtEndEnvironment{proposition}{\par\smallskip}
\newtheorem{proposition}{Proposition}
\newtheorem{corollary}{Corollary}
\newtheorem{remark}{Remark}
\newtheorem{theorem}{Theorem}
\newtheorem{definition}{Definition}

\DeclareMathOperator{\supp}{supp}

\newcommand{\LL}{\mathrm{L}}
\newcommand{\MM}{\mathrm{M}}
\newcommand{\UU}{\mathrm{U}}

\begin{document}

\title{A Geometric Witness Framework for Signed Multivariate Tail-Dependence Compatibility: Asymptotic Structure and Finite-Threshold Synthesis}
\author{Janusz Milek\\
Independent Researcher\\
\texttt{janusz.milek@alumni.ethz.ch}}
\date{}
\maketitle

\begin{abstract}
We study multivariate tail-dependence compatibility for complete and partial
signed tail families, treating lower-tail, upper-tail, and mixed configurations
in one geometric witness representation indexed by active coordinate sets and
sign patterns. For a complete signed tail family, witness generator weights
\(w=(w_{I,\sigma})\) give a linear incidence parametrization and are recovered
by explicit triangular inversion. Excluding the geometric scale \(p_0\), the
complete case uses \(3^d-1\) generator weights, matching the number of complete
signed tail coefficients; for partial specifications, only selected target
coefficients need be prescribed. At a fixed threshold \(p_0\in(0,\tfrac12)\),
the inversion identifies the normalized noncentral ternary cell masses of any
realizing copula. Hence finite-threshold compatibility is characterized by
nonnegative recovered generator weights, singleton normalization, and the
residual central-mass constraint.

This yields a complete M\"obius-type synthesis within the witness framework. If
the recovered increments are nonnegative and singleton normalization holds,
\(S(w)=\sum w\) determines the admissible finite-scale range, and every
admissible \(p_0\) gives an exact witness realization. In the canonical ray
geometry, such a realization preserves the same complete signed tail family
throughout \(0<p\le p_0\). Thus the primary object is the complete signed tail
family \(\lambda\): it is realized at every admissible finite scale and can be
carried along witness copulas with \(p_0\downarrow0\).

Partial, noisy, or inconsistent specifications are treated through
linear-feasibility and weighted-\(\ell^1\) recovery problems in the same
parametrization. The representation separates the \(p_0\)-free incidence/M\"obius
layer from finite-threshold realization and provides tools for realization,
simulation, calibration, completion, repair, and scenario design.
\end{abstract}

\medskip
\noindent\textit{MSC 2020.}


\section{Introduction}

Tail-dependence coefficients\footnote{For orientation, in the all-lower case
this refers to normalized joint-tail probabilities for the variables indexed by \(J\),
of the form \(p^{-1}\Pr(U_j\le p \text{ for all } j\in J)\), equivalently
\(C_J(p,\ldots,p)/p\) for the \(J\)-marginal copula, or to their limits as
\(p\downarrow0\) when such limits are considered. The present paper uses the
corresponding signed lower-/upper-/mixed finite-threshold readouts, made
explicit below through ternary tail totals.} are standard descriptors of extremal dependence
in copula theory and multivariate dependence modelling; see, for instance,
\citet{Joe2014} and \citet{Nelsen2006}. They also arise in applications such
as stress testing and risk aggregation \citep{MFE2015}. The signed setting considered here extends this viewpoint to lower-tail,
upper-tail, and mixed co-movements, in both bivariate and multivariate form. This is particularly relevant in
areas such as environmental risk, where large deviations in either direction
from a reference state may be harmful and where joint excursions can amplify
their effects. The main question studied here is mathematical: whether a
prescribed family of extremal dependence coefficients is \emph{compatible},
that is, whether there exists a multivariate model with continuous margins
realizing the target coefficients.

In the classical unsigned upper-tail setting, compatibility admits elegant
characterization based on Bernoulli-type constructions due to Embrechts, Hofert,
and Wang~(2016), abbreviated here as \emph{EHW}, and
subsequent developments \citep{KrauseScherer2018}. In broader multivariate copula settings, upper-
and lower-tail coefficients have also been studied as structural summaries of
joint extremal behaviour \citep{DeLucaRivieccio2012}, while extreme-value
copulas provide a classical asymptotic benchmark for dependence between rare
events \citep{GudendorfSegers2010}.  However, the signed multivariate setting considered here is genuinely
different. Lower-tail, upper-tail, and mixed configurations must be treated
simultaneously, and finite-grid realizations depend on the threshold. This
creates a mismatch between the tail-level coefficients one wants to prescribe
and the finite-scale objects produced by constructive realizations.

To resolve this mismatch, we introduce a constructive witness representation
for signed multivariate tail dependence. The tail level is described by
witness weights indexed by active coordinate sets and sign patterns, while at
finite threshold the same weights determine ternary cell masses on a partition
of the unit hypercube. This yields an exact linear \(p_0\)-free description of
complete signed tail families, explicit triangular inversion, and a direct
route to linear programming, implementation, and finite-scale simulation.

A finite-threshold realization naturally lives on the ternary state space
\[
\{\LL,\MM,\UU\}^d,
\]
where \(\LL\), \(\MM\), and \(\UU\) denote lower-tail, middle, and upper-tail
regions. Throughout the paper, \(p_0\in(0,\tfrac12)\) is the finite tail scale
used to partition each marginal axis into \([0,p_0]\), \([p_0,1-p_0]\), and
\([1-p_0,1]\), with overlaps only at endpoints and hence as a ternary
partition modulo null sets. The corresponding ternary cell probabilities
\[
(q_a)_{a\in\{\LL,\MM,\UU\}^d}
\]
belong to this finite geometric scale. By contrast, the prescribed signed tail
family is a tail-level object. To connect these two levels, we use witness
weights \(w_{I,\sigma}\), indexed by active coordinate sets and sign patterns.
They arise as the recovered increments of the triangular inversion and at
the same time serve as the basic parameters of finite-threshold realization,
through
\[
q_a = p_0\,w_{I(a),\sigma(a)}
\qquad\text{for every noncentral cell } a.
\]
Thus signed tail coefficients are linear incidence sums of witness weights,
while finite-\(p_0\) realization enters through the ternary masses and the
residual central mass.

Figure~\ref{fig:2d-flow} records the relation between the three objects used throughout the paper: witness weights $w$, the complete signed tail family $\lambda$, and the finite-threshold ternary masses $q$. The tail-level map $w \mapsto \lambda$ is inverted in the complete case by triangular inversion, whereas the finite-threshold relation between $w$ and $q$ is read in reverse only after fixing $p_0$.

This tail-level / finite-threshold split is useful in three ways. It gives a concrete witness representation rather than a purely existential compatibility statement. It yields an exact linear parametrization of complete signed tail families in witness weights. And it leads directly to computation: complete specifications admit explicit triangular inversion, while partial, noisy, or inconsistent specifications lead to linear feasibility or weighted approximation problems.

\par\vspace{\baselineskip}

\begin{figure}[H]
\centering
\includegraphics[width=1\textwidth]{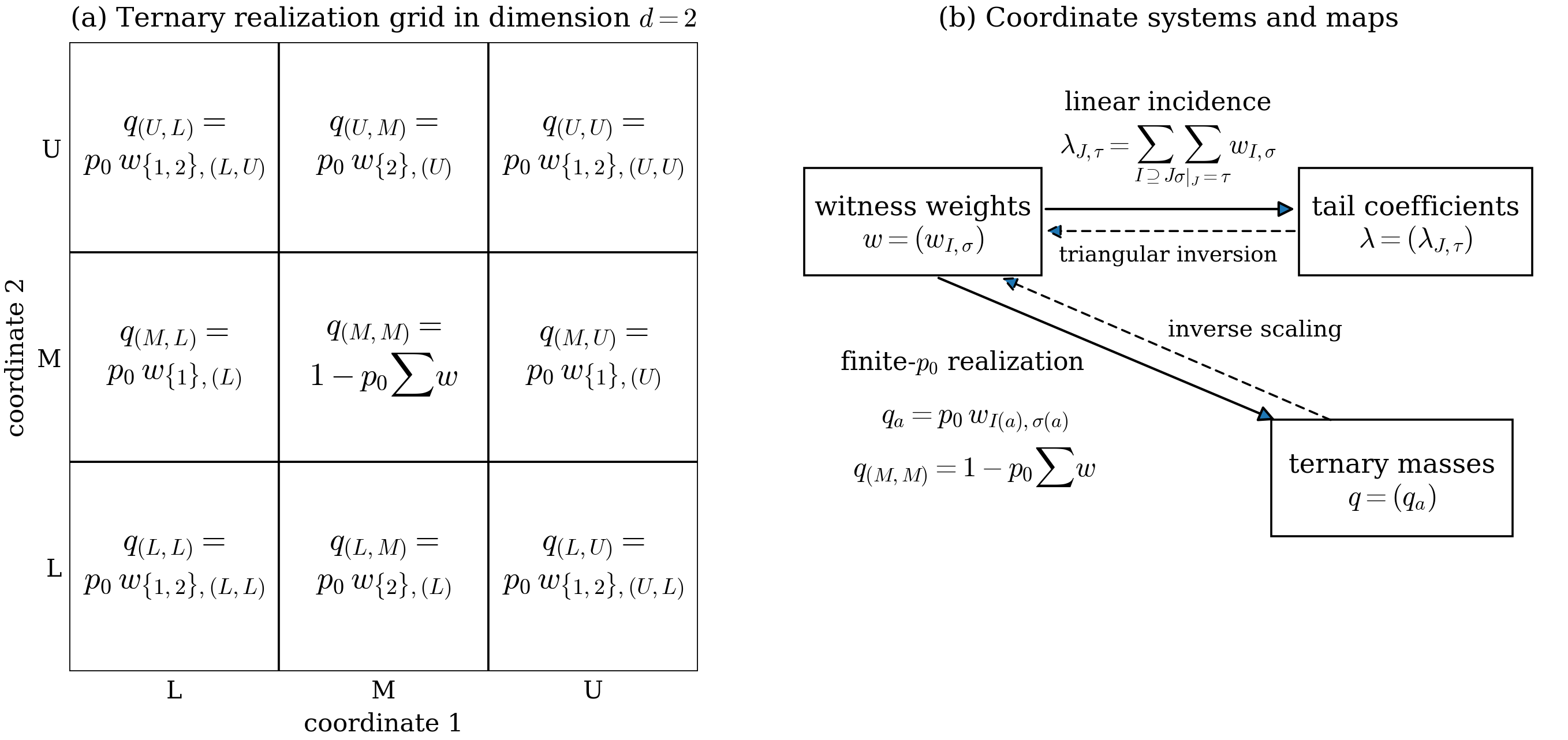}
\setlength{\abovecaptionskip}{0pt}
\caption{Left: the \(d=2\) finite-\(p_0\) ternary grid, where every noncentral
cell corresponds to a unique generator \((I,\sigma)\), so
\(q_a=p_0\,w_{I(a),\sigma(a)}\) for \(a\neq(\MM,\MM)\). Right: the three linked objects used throughout the paper. Witness weights \(w\)
are the tail-level parameters, signed tail coefficients \(\lambda\) are
obtained from \(w\) by linear incidence, and the finite-threshold ternary
masses \(q\) are obtained from \(w\) by finite-threshold realization,
equivalently by rescaling the noncentral generators at scale \(p_0\).
In the complete case, \(\lambda\mapsto w\) is recovered by triangular
inversion, while for fixed \(p_0\) the noncentral masses determine \(w\)
by inverse rescaling.}
\label{fig:2d-flow}
\end{figure}

\paragraph{Positioning relative to existing constructive approaches.}
The present framework is not simply a signed extension of the Bernoulli-based
upper-tail construction of Embrechts, Hofert, and Wang~(2016). In the pure
upper-tail complete regime, the witness incidence system reduces to the
familiar Boolean-type incidence structure and its inversion. However, the signed
setting studied here is broader: lower-tail, upper-tail, and inactive
coordinates must be handled simultaneously, and finite-threshold realizations
live on a ternary partition. This is why we distinguish between tail-level
coefficients and their finite-\(p_0\) realizations.

Earlier work by the author considered application-driven constructive atomic
representations for multivariate tail structures in the context of risk
aggregation. The present paper develops a different mathematical framework.
Its basic objects are witness generators indexed by active sets and sign
patterns, giving a \(3^d-1\) noncentral generator system, a linear
\(p_0\)-free parametrization of complete signed tail families, and an explicit
triangular recovery scheme. In this sense, the relation to the earlier work is
one of continuity of viewpoint rather than identity of method.

The link with M\"obius inversion is structural: the backward recovery developed
below is a triangular inversion on the partially ordered family of signed
active-set indices, so it recovers the increments of a complete signed tail
family. At the same time, the construction remains combinatorially controlled.
Instead of using a much larger family of candidate signed extremal patterns,
it works with a tractable geometric generator language that still captures
rich signed dependence configurations. The point is therefore not only an
analogy with M\"obius inversion, but an explicit constructive synthesis
linking tail-level structure, finite-\(p_0\) realization, and the
computational formulations developed later in the paper. Appendix~\ref{app:poset-moebius} makes this structural link explicit: the signed
incidence system is the zeta matrix of a nonzero signed ternary poset, every
interval of this poset is Boolean, and the complete-case recovery is the corresponding M\"obius inversion.

\paragraph{Asymptotic viewpoint.}
The complete signed tail family $\lambda$ is the primary asymptotic object.
At any admissible finite scale $p_0$, the canonical ray realization already
preserves that same family throughout the whole threshold range
$0<p\le p_0$. The vanishing-threshold viewpoint therefore does not serve to
recover $\lambda$ below a fixed $p_0$; rather, it expresses the freedom to
realize the same object along families of witness copulas with $p_0\downarrow0$.
This is possible because the incidence relation $w\mapsto\lambda$ is $p_0$-free,
whereas the geometric scale enters only through the finite-threshold realization,
exact margins, and the residual central mass.

\paragraph{Main contributions.}
The paper makes five main contributions. It introduces a signed geometric
witness language indexed by active sets and sign patterns, where ``active''
means that the coordinate lies in one of the two tail regions \(\LL\) or
\(\UU\), while inactive coordinates remain in the middle region \(\MM\). It
shows that complete signed multivariate tail families admit an exact linear
\(p_0\)-free parametrization in witness weights. It proves a complete-case
synthesis theorem: at a fixed finite threshold, triangular inversion recovers
the normalized noncentral ternary cell masses of any realizing copula, so
complete compatibility is characterized exactly by nonnegative recovered
increments, singleton normalization, and residual central mass. It then gives
a finite-threshold admissibility theorem: once the recovered tail-level
increments are nonnegative and singleton normalization holds, the admissible
scale range is explicit, and every admissible \(p_0\) yields an exact
realization of the same complete signed tail family. Finally, it formulates
partial, noisy, and inconsistent recovery problems as linear feasibility and
weighted-approximation problems in the same witness representation, and
connects this algebraic core to implementation, simulation, and expert-driven
calibration.

\paragraph{Practical relevance.}
Beyond its mathematical role, the witness copula provides an interpretable
scenario language for dependence. Its generators encode explicit joint
lower-tail, upper-tail, and mixed extremal behaviours, while allowing the remaining coordinates to vary within the middle region. This is useful when dependence is described through stress scenarios, expert
judgement, or partial structural information about extremal dependence rather
than through a fully specified parametric copula family. Within the same witness parametrization, such inputs can be calibrated, completed, or repaired by the linear methods developed later in the paper.

\paragraph{Structural conventions used throughout.}
All noncentral generators are normalized within one common framework.
The parameter \(p_0\) is shared across coordinates and fixes the geometric
realization scale, each noncentral generator is tail-scaled to total mass
\(p_0\), the central cell is treated as a residual component rather than as
another generator, and singleton coefficients are kept as one-dimensional
normalization constraints. Together, these conventions make the later
incidence matrix binary and the complete witness inversion triangular.

The paper is organized as follows. Section~\ref{sec:construction} introduces the witness language
of signed cells, generator indices, and the canonical ray geometry. Section~\ref{sec:tailmap}
develops the tail-level layer: the single-generator impact rule, the linear
incidence map $\lambda=Aw$, and the triangular inversion for complete
specifications. Section~\ref{sec:margins} passes to finite-$p_0$ ternary masses, develops the
corresponding tail-event combinatorics, and proves the central fixed-scale
invariance result showing that the canonical ray witness preserves the same
complete signed tail family throughout the whole range $0<p\le p_0$.
Section~\ref{sec:lp} extends the same witness parametrization to partial, noisy, and
inconsistent targets through linear programming. Section~\ref{sec:simulation} presents the
simulation layer, the five-dimensional benchmark, and the validation results.
Section~\ref{sec:formal-shell} gives the formal measure-theoretic shell showing that the witness
construction defines a copula. Section~\ref{sec:asymptotic} then draws the main consequences:
complete finite-threshold synthesis, the admissible scale range, and the
vanishing-threshold asymptotic interpretation. The discussion and conclusions
summarize the scope, limits, and significance of the results.

\section{Witness language, signed cells, and structural conventions}
\label{sec:construction}

This section fixes the geometric language used in the main line of the paper.
We fix a ternary partition of \([0,1]^d\), index each noncentral signed cell by
an active set and a sign pattern, and use those indices as the basic witness indexing system. Through Section~\ref{sec:benchmark}, the argument is then
signed-cell combinatorics, linear algebra, and finite-\(p_0\)
realization.

\subsection{Ternary partition and signed cells}
Throughout this section, let
\[
[d]:=\{1,\dots,d\}.
\]

Fix a geometric scale
\[
p_0\in(0,1/2).
\]
Define the three basic intervals
\[
L_{p_0}:=[0,p_0],
\qquad
M_{p_0}:=[p_0,1-p_0],
\qquad
U_{p_0}:=[1-p_0,1].
\]
These intervals overlap only at endpoints, so the resulting ternary decomposition
is a genuine partition modulo null sets.

For a nonempty active set \(I\subseteq[d]\), a sign pattern
\(\sigma\in\{\LL,\UU\}^{I}\)\footnote{Here and throughout, for a finite index set \(I\), the notation \(\{\LL,\UU\}^{I}\) denotes the set of all sign patterns indexed by \(I\), equivalently all functions \(\sigma:I\to\{\LL,\UU\}\). Thus \(\sigma=(\sigma_i)_{i\in I}\) assigns to each active coordinate \(i\in I\) either \(\LL\) or \(\UU\).}, and \(i\in[d]\), define
\[
A_i^{(I,\sigma)} :=
\begin{cases}
L_{p_0}, & i\in I \text{ and } \sigma_i=\LL,\\[1mm]
U_{p_0}, & i\in I \text{ and } \sigma_i=\UU,\\[1mm]
M_{p_0}, & i\notin I.
\end{cases}
\]
The corresponding signed ternary cell is
\[
B_{I,\sigma}^{(p_0)}
:=
\prod_{i=1}^d A_i^{(I,\sigma)}
\subseteq [0,1]^d.
\]

\begin{definition}[Central cell]
The central cell is
\[
B_0^{(p_0)} := M_{p_0}^d.
\]
\end{definition}

\begin{remark}
Geometrically, a generator \((I,\sigma)\) places the coordinates in \(I\) in
the lower-tail or upper-tail regions \(\LL\) and \(\UU\) according to the signs in
\(\sigma\), while all coordinates outside \(I\) remain in the middle region \(\MM\).
In dimension \(3\), this yields a partition into the central cube and
\(3^3-1=26\) noncentral signed components.
\end{remark}

\subsection{Generator indices and the noncentral cell correspondence}
The primitive noncentral generators are indexed by active coordinate sets and
sign patterns.

\begin{definition}[Generator index set]
Define
\[
\mathcal G_d
:=
\{(I,\sigma): \emptyset\neq I\subseteq[d],\ \sigma\in\{\LL,\UU\}^{I}\}.
\]
\end{definition}

\begin{remark}
Its cardinality is
\[
|\mathcal G_d|
=
\sum_{k=1}^d \binom{d}{k}2^k
=
3^d-1.
\]
Thus the number of primitive noncentral generators coincides with the number of
noncentral ternary cells in \(\{\LL,\MM,\UU\}^d\). The correspondence is
canonical: a noncentral ternary cell determines its active set \(I\subseteq[d]\)
by the coordinates not equal to \(\MM\), and its signs on \(I\) determine a
unique pattern \(\sigma\in\{\LL,\UU\}^I\); conversely, every
\((I,\sigma)\in\mathcal G_d\) determines exactly one noncentral signed cell
\(B^{(p_0)}_{I,\sigma}\).
\end{remark}

\begin{remark}
The sign pattern is indexed only on the active set \(I\). Inactive coordinates
lie in the middle region and therefore do not carry tail signs. This removes
the redundancy that would arise from indexing generators by full
\(d\)-dimensional sign vectors.
\end{remark}

\subsection{Support geometry of a generator}

Combinatorially, a generator is uniquely identified by its signed cell.
In this paper, however, the in-cell geometry is not an additional degree of
freedom. Unless explicitly stated otherwise, a witness generator means the
canonical shared-factor generator below. The following construction is therefore
the fixed generator geometry used throughout the paper.

\begin{proposition}[Canonical noncentral generators and the central component]
\label{prop:canonical-generators}
For every \((I,\sigma)\in\mathcal G_d\) and every \(p_0\in(0,1/2)\), let
\(Z\sim \mathrm{Unif}[0,1]\), and independently let \((V_i)_{i\notin I}\) be
i.i.d.\ \(\mathrm{Unif}[0,1]\). Define a random vector
\(U^{(I,\sigma)}=(U_1,\dots,U_d)\) by
\[
U_i=
\begin{cases}
p_0 Z, & i\in I \text{ and } \sigma_i=\LL,\\[1mm]
1-p_0 Z, & i\in I \text{ and } \sigma_i=\UU,\\[1mm]
p_0+(1-2p_0)V_i, & i\notin I.
\end{cases}
\]
Then \(U^{(I,\sigma)}\) is supported on \(B_{I,\sigma}^{(p_0)}\), each active
coordinate is uniform on its prescribed tail interval, and each inactive
coordinate is uniform on \(M_{p_0}\).

For the central cell, define
\[
U_i^{(0)}:=p_0+(1-2p_0)V_i,\qquad i=1,\dots,d,
\]
with \(V_1,\dots,V_d\) i.i.d.\ \(\mathrm{Unif}[0,1]\). Then
\(U^{(0)}=(U_1^{(0)},\dots,U_d^{(0)})\) is supported on the central cell
\(B^{(p_0)}_0\) and has uniform one-dimensional marginals on \(M_{p_0}\).
\end{proposition}

\begin{proof}
For active coordinates, the construction places \(U_i\) in the prescribed
tail interval and makes it uniform there because \(Z\) is uniform on
\([0,1]\). For inactive coordinates, \(U_i\) is uniform on \(M_{p_0}\) by
construction. Hence the support and one-dimensional marginal statements
follow immediately. The central construction is the special case in which all
coordinates remain in \(M_{p_0}\).
\end{proof}

\begin{remark}[Geometric interpretation of the canonical generators]
The canonical construction separates two types of directions. The active
coordinates \(i\in I\) are coupled through the single shared factor \(Z\),
which creates one common extremal dependence direction: on the active
coordinates, the support follows the ray joining the corresponding outer
vertex of the unit hypercube to the matching vertex of the central cell
\(M_{p_0}^d\). By contrast, each inactive coordinate \(i\notin I\) is driven
by its own variable \(V_i\) and therefore moves freely inside the middle
region \(M_{p_0}\). Thus the canonical support combines one joint tail
direction with \(d-|I|\) neutral middle directions. In particular, larger
active sets correspond to more concentrated geometric generators and to
higher-order extremal interaction.
\end{remark}

\begin{remark}
For the canonical choice, the support dimension is \(1+(d-|I|)\): in
dimension \(3\), generators with \(|I|=3\), \(|I|=2\), and \(|I|=1\) are
supported on line segments, plates, and rectangular boxes, respectively,
with the active part aligned along the corresponding outer-corner-to-central-vertex
ray. Figure~\ref{fig:gwc-27cells-png-before-refs} is drawn in this canonical geometry. In particular, many
witness generators are geometrically singular with respect to
full-dimensional Lebesgue measure on \([0,1]^d\). This is a feature of the
construction, not a defect: the witness language is meant to encode explicit
extremal mechanisms rather than smooth density models.
\end{remark}

\begingroup
\setlength{\intextsep}{0.2\baselineskip}
\setlength{\abovecaptionskip}{3pt}
\setlength{\belowcaptionskip}{0pt}
\begin{figure}[H]
\centering
\includegraphics[width=0.77\textwidth,clip]{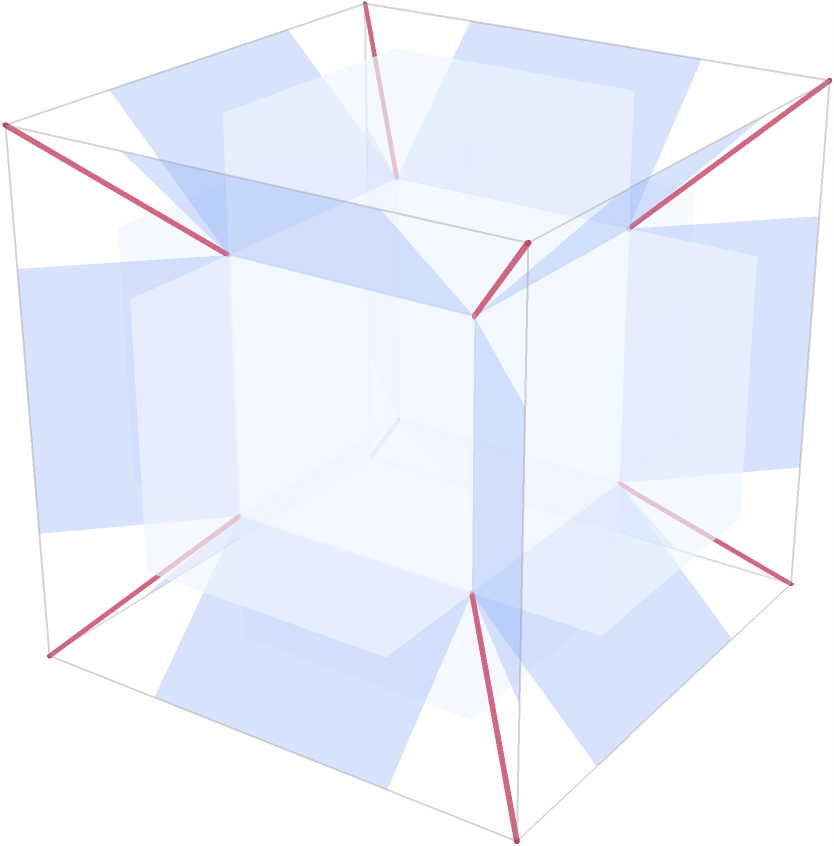}
\caption{A three-dimensional witness partition into the central cube and
\(3^3-1=26\) noncentral signed geometric components, shown in the canonical
\(d=3\) geometry associated with the shared-factor generator choice of
Proposition~\ref{prop:canonical-generators}. In this geometry, active-set
sizes \(|I|=3,2,1\) appear as line segments, plates, and rectangular boxes,
respectively. Rendered with Wolfram Mathematica 14.3.}
\label{fig:gwc-27cells-png-before-refs}
\end{figure}
\endgroup

The remark above describes the internal geometry of the canonical support,
while the following table below summarizes the cell-level geometry of witness generators
according to the active-set size $k=|I|$.
The volume column refers to the signed cells $B_{I,\sigma}^{(p_0)}$, not to
the lower-dimensional canonical supports inside them.
Here ``inward'' and ``outward'' refer to cell neighbors obtained by changing
one coordinate state by a single step in the ternary alphabet \(\LL/\MM/\UU\),
toward or away from the central cell \(M_{p_0}^d\).
Each row describes a single generator type with fixed active-set size
\(k=|I|\); in particular, the column ``affected coeffs.'' refers to one
individual generator \((I,\sigma)\), not to the whole layer of generators of
order \(k\).

\begin{table}[H]
\centering
\small
\setlength{\tabcolsep}{4pt}
\caption{Cell-level geometry of witness generators of order $k=|I|$.}
\label{tab:cell-geometry}
\begin{tabular}{c c c c c c c c}
\hline
$k$ &
\shortstack{count} &
\shortstack{support\\dim.} &
\shortstack{$M_{p_0}^d$ contact\\dim.} &
\shortstack{inward\\nbrs.} &
\shortstack{outward\\nbrs.} &
\shortstack{cell\\volume} &
\shortstack{affected\\coeffs.\\per generator} \\
\hline
$1$
& $\binom{d}{1}2$
& $d$
& $d-1$
& $1$
& $2(d-1)$
& $p_0(1-2p_0)^{d-1}$
& $1$ \\
$2$
& $\binom{d}{2}2^2$
& $d-1$
& $d-2$
& $2$
& $2(d-2)$
& $p_0^2(1-2p_0)^{d-2}$
& $3$ \\
$\vdots$
& $\vdots$
& $\vdots$
& $\vdots$
& $\vdots$
& $\vdots$
& $\vdots$
& $\vdots$ \\
$k$
& $\binom{d}{k}2^k$
& $d-k+1$
& $d-k$
& $k$
& $2(d-k)$
& $p_0^k(1-2p_0)^{d-k}$
& $2^k-1$ \\
$\vdots$
& $\vdots$
& $\vdots$
& $\vdots$
& $\vdots$
& $\vdots$
& $\vdots$
& $\vdots$ \\
$d$
& $2^d$
& $1$
& $0$
& $d$
& $0$
& $p_0^d$
& $2^d-1$ \\
\hline
total
& $3^d-1$
& --
& --
& --
& --
& $1-(1-2p_0)^d$
& -- \\
\hline
\end{tabular}
\end{table}

\vspace{1\baselineskip}

\subsection{Structural conventions}

The witness framework uses one common threshold \(p_0\in(0,1/2)\) across all
generators, and the in-cell geometry is the canonical ray geometry of
Proposition~\ref{prop:canonical-generators}. Every noncentral generator is
tail-scaled to total mass \(p_0\), which is exactly what makes the later
incidence matrix binary. Thus the generator-level degrees of freedom are the
witness weights, not alternative in-cell geometries.
The central cell \(B_0^{(p_0)}=M_{p_0}^d\) is treated as a residual
central component rather than as another generator: it carries no active set and
no sign pattern, it contributes to total probability and exact margins, but it
does not contribute to signed tail coefficients.

Singleton coefficients are nevertheless retained in the same witness
parametrization; for example, \(\lambda_{\{i\},\LL}\) and \(\lambda_{\{i\},\UU}\)
record the one-dimensional lower-tail and upper-tail normalization for
coordinate \(i\). They are not introduced as genuine multivariate extremal interaction
coefficients; instead, their role is to encode one-dimensional tail
normalization and hence the exact marginal constraints. This is also why generators with \(|I|=1\) remain part of the same witness
construction, even though they affect only singleton coefficients. In Table~\ref{tab:cell-geometry}, the row \(k=1\) therefore encodes only singleton normalization.

\begin{definition}[Witness weights]
\label{def:wweights}
A witness weight system is a family
\[
w=(w_{I,\sigma})_{(I,\sigma)\in\mathcal G_d},
\qquad
w_{I,\sigma}\ge 0.
\]
Its associated central coefficient is
\[
m_0 := 1-p_0\sum_{(I,\sigma)\in\mathcal G_d} w_{I,\sigma}.
\]
\end{definition}

\begin{remark}[Roles of the main objects]
A signed cell \(B^{(p_0)}_{I,\sigma}\) is a geometric region of the ternary
partition, its support is the region or lower-dimensional geometric set used by
the corresponding canonical generator, its volume is a Euclidean quantity, its mass is
probabilistic, and its witness weight \(w_{I,\sigma}\) is a tail-level
structural parameter. By contrast, the signed coefficients
\(\lambda_{J,\tau}\), introduced in Section~\ref{sec:tailmap}, are observable
quantities induced by the whole weighted superposition. Once \(p_0\) is fixed,
the ternary cell masses \(q_a\) belong to the realization level rather than to
the tail-level parametrization.
\end{remark}

\subsection{What this section does not yet do}

This section fixes the language and indexing. The next sections use only
signed cells, generator signatures, linear superposition, backward
inversion, and the finite-\(p_0\) passage to ternary cell masses. The formal copula
construction is deferred to Section~\ref{sec:formal-shell} so that the main
narrative can stay focused on the constructive line of the paper.
What matters later is that in this fixed canonical choice the active coordinates
move along a common ray. Once finite-threshold coefficients are introduced in
Section~\ref{sec:margins}, this geometry will lead directly to a fixed-scale invariance
property of the resulting signed tail family.

\section{Single-generator impact and the linear witness map}
\label{sec:tailmap}

\subsection{Signed tail coefficients as structural observables}

We now pass from signed cells to the induced coefficient system. Once the
generator indices are fixed, each signed coefficient is determined by which generators extend the prescribed signed subpattern and how their weights accumulate. This
observation underlies the incidence structure of the witness framework.
As in the Introduction, \(\lambda\) denotes the observable tail-level family
and \(w\) the witness weights. Once a geometric scale is
chosen in Section~\ref{sec:q-representation}, the corresponding finite-threshold
ternary cell masses will be denoted by \(q\).

\begin{definition}[Complete signed tail family]
\label{def:signed-tail-coefficient}
A complete signed tail family is a collection
\[
\lambda=(\lambda_{J,\tau})_{\emptyset\neq J\subseteq[d],\ \tau\in\{\LL,\UU\}^{J}}
\]
indexed by all nonempty active sets and all sign patterns on those sets.
At this stage, \(\lambda\) is simply the complete target family of signed tail coefficients,
one entry for each admissible pair \((J,\tau)\).
\end{definition}

\begin{remark}[Tail-level versus asymptotic language]
\label{rem:tail-level-vs-asymptotic}
These coefficients belong to the tail-level description of the witness
framework. Their finite-threshold interpretation in terms of ternary cell masses
and tail totals is developed in Section~\ref{sec:tail-event-combinatorics}, and
their asymptotic role is discussed in Section~\ref{sec:asymptotic}. This separation of tail-level, finite-threshold, and asymptotic viewpoints
matters when comparing the present framework with asymptotic copula theory
\citep{GudendorfSegers2010,DeLucaRivieccio2012}.
\end{remark}

Within the witness framework, these coefficients are realized as incidence
sums of witness weights: \(\lambda_{J,\tau}\) records the total weight of all
generators whose active set contains \(J\) and whose sign pattern restricts to
\(\tau\) on \(J\). Thus a single generator may contribute to several signed tail
coefficients. At the finite-threshold level developed in Section~\ref{sec:tail-event-combinatorics},
the same coefficients appear as normalized signed joint-tail event totals on
the specified coordinates, with the remaining coordinates left unrestricted.

The complete signed tail family is therefore indexed by the same type of
objects as the witness weights, namely by nonempty active sets together with
sign patterns on those active sets. In particular, the number of complete
signed tail coefficients is also \(3^d-1\).

\subsection{Contribution of a single generator to signed tail coefficients}

With the coefficient family fixed, the next step is to record the contribution of a single generator.

\begin{definition}[Signed tail signature]
\label{def:signed-tail-signature}
For each \((I,\sigma)\in\mathcal G_d\), define its signed tail signature across
all admissible targets \((J,\tau)\) by
\[
\eta^{(I,\sigma)}_{J,\tau}
:=
\mathbf{1}\{J\subseteq I,\ \sigma|_J=\tau\},
\qquad
\emptyset\neq J\subseteq[d],\ \tau\in\{\LL,\UU\}^J.
\]
\end{definition}

\begin{remark}
For a fixed generator \((I,\sigma)\), this incidence profile is binary: it
equals \(1\) exactly on those target patterns \((J,\tau)\) obtained by
restricting \((I,\sigma)\), and \(0\) otherwise.
\end{remark}

\paragraph{Additivity and weighted superposition.}
The single-generator signature already contains the whole mechanism: local
behavior is 0/1, and global behavior is linear. Once generators are combined
with nonnegative weights, every signed coefficient is simply the sum of the
weights of the generators that contain the requested signed subpattern. In
matrix form, one may therefore write \(\lambda=Aw\), where \(A\) is the
signed incidence matrix induced by this extension rule.

\begin{theorem}[Linear witness map]
\label{thm:w-to-lambda}
Within the witness class, the complete signed tail family is given by
\[
\lambda_{J,\tau}
=
\sum_{I\supseteq J}
\ \sum_{\sigma:\,\sigma|_J=\tau}
w_{I,\sigma},
\qquad
\emptyset\neq J\subseteq[d],\ \tau\in\{\LL,\UU\}^{J}.
\]
Equivalently, \(\lambda_{J,\tau}\) is obtained by summing the weights of
exactly those generators whose active pattern extends the requested signed
subpattern \((J,\tau)\). In matrix form, \(\lambda=Aw\), where the signed incidence matrix \(A\) has entries
\[
A_{(J,\tau),(I,\sigma)}=\mathbf 1\{J\subseteq I,\ \sigma|_J=\tau\}.
\]
In particular, the map \(w\mapsto \lambda\) is linear and independent of \(p_0\).
\end{theorem}

\begin{proof}
By weighted superposition, every generator \((I,\sigma)\) contributes its own
signature \(\eta^{(I,\sigma)}_{J,\tau}\) multiplied by the weight
\(w_{I,\sigma}\). Summing over all generators therefore gives
\[
\lambda_{J,\tau}
=
\sum_{(I,\sigma)\in\mathcal G_d}
w_{I,\sigma}\,\eta^{(I,\sigma)}_{J,\tau}.
\]

Using Definition~\ref{def:signed-tail-signature}, this becomes
\[
\lambda_{J,\tau}
=
\sum_{(I,\sigma)\in\mathcal G_d}
w_{I,\sigma}\,\mathbf{1}\{J\subseteq I,\ \sigma|_J=\tau\},
\]
so the sum keeps exactly those generators extending \((J,\tau)\) and
excludes all others. Hence
\[
\lambda_{J,\tau}
=
\sum_{I\supseteq J}
\ \sum_{\sigma:\,\sigma|_J=\tau}
w_{I,\sigma}.
\]
The right-hand side contains no occurrence of \(p_0\), so the map is
\(p_0\)-free.
\end{proof}

A concrete \(d=3\) illustration of this extension rule, together with the
corresponding finite-threshold cell totals, is shown later in
Figure~\ref{fig:incidence-d3}.

\paragraph{Incidence matrix in dimension \(d=2\).}
Order both the coefficient vector and the witness vector by
\[
(1,L),\ (1,U),\ (2,L),\ (2,U),\ (12,LL),\ (12,LU),\ (12,UL),\ (12,UU).
\]
In this ordering, the forward incidence relations read
\begin{align*}
\lambda_{12,LL} &= w_{12,LL}, &
\lambda_{12,LU} &= w_{12,LU}, \\
\lambda_{12,UL} &= w_{12,UL}, &
\lambda_{12,UU} &= w_{12,UU}, \\
\lambda_{1,L} &= w_{1,L}+w_{12,LL}+w_{12,LU}, \qquad &
\lambda_{1,U} &= w_{1,U}+w_{12,UL}+w_{12,UU}, \\
\lambda_{2,L} &= w_{2,L}+w_{12,LL}+w_{12,UL}, \qquad &
\lambda_{2,U} &= w_{2,U}+w_{12,LU}+w_{12,UU}.
\end{align*}
These equations are encoded by \(\lambda=A_2\,w\), with
\begingroup
\setlength{\arraycolsep}{4pt}
\renewcommand{\arraystretch}{0.8}
\normalsize\[
A_2 =
\left(
\setlength{\arraycolsep}{0pt}
\begin{array}{*{8}{w{c}{2em}}}
1 & 0 & 0 & 0 & 1 & 1 & 0 & 0 \\
0 & 1 & 0 & 0 & 0 & 0 & 1 & 1 \\
0 & 0 & 1 & 0 & 1 & 0 & 1 & 0 \\
0 & 0 & 0 & 1 & 0 & 1 & 0 & 1 \\
0 & 0 & 0 & 0 & 1 & 0 & 0 & 0 \\
0 & 0 & 0 & 0 & 0 & 1 & 0 & 0 \\
0 & 0 & 0 & 0 & 0 & 0 & 1 & 0 \\
0 & 0 & 0 & 0 & 0 & 0 & 0 & 1
\end{array}
\right).
\]
Thus \(A_2\) is upper triangular with unit diagonal, hence invertible, and
\begingroup
\setlength{\arraycolsep}{4pt}
\renewcommand{\arraystretch}{0.8}
\normalsize
\[
w = A_2^{-1}\lambda,
\qquad
A_2^{-1} =
\left(
\setlength{\arraycolsep}{0pt}
\begin{array}{*{8}{w{c}{2em}}}
1 & 0 & 0 & 0 & -1 & -1 & 0 & 0 \\
0 & 1 & 0 & 0 & 0 & 0 & -1 & -1 \\
0 & 0 & 1 & 0 & -1 & 0 & -1 & 0 \\
0 & 0 & 0 & 1 & 0 & -1 & 0 & -1 \\
0 & 0 & 0 & 0 & 1 & 0 & 0 & 0 \\
0 & 0 & 0 & 0 & 0 & 1 & 0 & 0 \\
0 & 0 & 0 & 0 & 0 & 0 & 1 & 0 \\
0 & 0 & 0 & 0 & 0 & 0 & 0 & 1
\end{array}
\right)
\]

Equivalently, the inverse recovery reads
\begin{align*}
w_{12,LL} &= \lambda_{12,LL}, &
w_{12,LU} &= \lambda_{12,LU}, \\
w_{12,UL} &= \lambda_{12,UL}, &
w_{12,UU} &= \lambda_{12,UU}, \\
w_{1,L} &= \lambda_{1,L}-\lambda_{12,LL}-\lambda_{12,LU}, \qquad &
w_{1,U} &= \lambda_{1,U}-\lambda_{12,UL}-\lambda_{12,UU}, \\
w_{2,L} &= \lambda_{2,L}-\lambda_{12,LL}-\lambda_{12,UL}, \qquad &
w_{2,U} &= \lambda_{2,U}-\lambda_{12,LU}-\lambda_{12,UU}.
\end{align*}
\endgroup
\endgroup

Thus the \(d=2\) case already exhibits both directions explicitly: the final four rows of \(A_2\),
corresponding to the order-\(2\) targets \((12,LL)\), \((12,LU)\), \((12,UL)\), and \((12,UU)\),
identify the pair generators directly, while the initial four singleton rows then collect what remains.
Equivalently, the displayed system shows both the forward incidence map
\(w\mapsto\lambda\) and the reverse recovery \(\lambda\mapsto w\) obtained by
triangular back-substitution. Larger explicit incidence tables,
including the complete signed \(d=3\) case and a pure upper-tail multivariate example
in dimension \(d=4\), are collected in Appendix~\ref{app:incidence-tables}.

\subsection{Triangular inversion in the complete case}
\label{subsec:inverse}

Having written the forward incidence map explicitly, we now read it in the reverse direction for complete signed families. The linear map in Theorem~\ref{thm:w-to-lambda} is triangular with respect to
the partial order given by set inclusion on active coordinate sets. This yields
an explicit backward recovery formula.

\begin{proposition}[Backward recovery of witness weights]
\label{prop:triangular-inversion}
Let \((\lambda_{J,\tau})\) be a complete signed tail family.
Then the witness weights are recovered recursively, by descending active-set
size \(|I|\) from \(d\) down to \(1\), via
\[
w_{I,\sigma}
=
\lambda_{I,\sigma}
-
\sum_{K\supsetneq I}
\ \sum_{\rho:\,\rho|_I=\sigma}
w_{K,\rho},
\qquad
(I,\sigma)\in\mathcal G_d.
\]
In particular, for a complete signed tail family, the witness weight system
\(w=(w_{I,\sigma})\) is uniquely determined within the witness framework.
\end{proposition}

\begin{proof}
Fix \((I,\sigma)\in G_d\). Applying Theorem~\ref{thm:w-to-lambda} with \((J,\tau)=(I,\sigma)\) gives
\[
\lambda_{I,\sigma}
=
\sum_{K\supseteq I}
\ \sum_{\rho:\,\rho|_I=\sigma}
w_{K,\rho}.
\]
The term \(K=I\) contributes exactly \(w_{I,\sigma}\), so
\[
\lambda_{I,\sigma}
=
w_{I,\sigma}
+
\sum_{K\supsetneq I}
\ \sum_{\rho:\,\rho|_I=\sigma}
w_{K,\rho}.
\]
Rearranging yields the stated formula. Since the right-hand side involves only
weights indexed by strict active-set supersets \(K\supsetneq I\), the recovery
proceeds uniquely by descending induction on \(|I|\), from \(d\) down to \(1\).
\end{proof}

The underlying combinatorial structure of this signed incidence system is
described in Appendix~\ref{app:poset-moebius}.

\begin{proposition}[Unsigned upper-tail restriction]
\label{prop:unsigned-upper-tail-restriction}
Consider the witness framework restricted to generators $(I,\sigma)\in G_d$
whose sign pattern is identically upper, that is, $\sigma\equiv \UU$ on $I$.
Then the signed active-set/sign-pattern incidence system collapses exactly to
the Boolean-lattice incidence system for complete unsigned upper-tail tail
families.

More precisely, under this restriction, each generator is identified uniquely
with its active set $I\subseteq[d]$, $I\neq\emptyset$, and for every nonempty
$J\subseteq[d]$ the complete upper-tail coefficient satisfies
\[
\lambda_J^{U}=\sum_{I\supseteq J} w_I^{U}.
\]
Hence the complete-case backward recovery reduces to the standard binary
M\"obius/back-substitution formula
\[
w_I^{U}
=
\sum_{K\supseteq I}(-1)^{|K|-|I|}\lambda_K^{U},
\qquad \emptyset\neq I\subseteq[d].
\]
In this sense, the present witness algebra contains the complete unsigned
upper-tail multivariate regime as a direct restriction.
\end{proposition}

\begin{proof}
Under the restriction \(\sigma\equiv \UU\), the generator index \((I,\sigma)\)
is completely determined by the active set \(I\), so the signed index family reduces to the family of
nonempty subsets of $[d]$. For a complete upper-tail event indexed by
$J\subseteq[d]$, a generator contributes if and only if all coordinates in $J$
are active and carry the upper sign, which under the restriction $\sigma\equiv \UU$
is equivalent to $J\subseteq I$. Therefore
\[
\lambda_J^{U}=\sum_{I\supseteq J} w_I^{U}.
\]
This is precisely the zeta-transform on the Boolean lattice of nonempty
subsets. Its inverse is the usual M\"obius inversion on that lattice, which
gives
\[
w_I^{U}
=
\sum_{K\supseteq I}(-1)^{|K|-|I|}\lambda_K^{U}.
\]
\end{proof}

\begin{remark}
Proposition~\ref{prop:unsigned-upper-tail-restriction} is more than a formal
analogy: in the complete upper-tail multivariate regime it is exactly the
Boolean-lattice zeta/M\"obius system.
\end{remark}

For \(d=3\), Figure~\ref{fig:hasse} in Appendix~\ref{app:hasse}
visualizes the signed active-set partial order behind the incidence map
\(\lambda=Aw\) and the direct inversion procedure: recovery proceeds from the
top layer \(|J|=3\) down to singleton indices.

\paragraph{Singleton coordinates.}
The singleton rows of the incidence map should be read as
normalization rows. They belong to the same coordinate system as the higher-order
signed coefficients, but their main role is to encode the one-dimensional
lower- and upper-tail normalization constraints that later become exact marginal
conditions.

\begin{corollary}[Complete-case validity of the tail-level witness parameters]
\label{cor:complete-tail-validity}
Let \((\lambda_{J,\tau})\) be a complete signed tail family, and let \(w\) be the
weights recovered by Proposition~\ref{prop:triangular-inversion}. Then
\((\lambda_{J,\tau})\) determines a valid tail-level witness parametrization if
and only if
\[
w_{I,\sigma}\ge 0
\qquad\text{for all }(I,\sigma)\in\mathcal G_d,
\]
together with the singleton normalization
\[
\lambda_{\{i\},\LL}=1,
\qquad
\lambda_{\{i\},\UU}=1,
\qquad i=1,\dots,d.
\]
\end{corollary}

\begin{proof}
The recovered weights are uniquely determined by Proposition~\ref{prop:triangular-inversion}. Nonnegativity is
necessary for any valid tail-level witness parametrization. The singleton
normalization is necessary because the singleton rows of the incidence map are
exactly the one-dimensional tail-normalization conditions.
Indeed, for any copula and any finite threshold \(p\in(0,1/2)\), one has
\(
\lambda^{(p)}_{\{i\},L}
=
p^{-1}\Pr(U_i\in[0,p])=1\) and 
\(
\lambda^{(p)}_{\{i\},U}
=
p^{-1}\Pr(U_i\in[1-p,1])=1
\)
for every \(i\in[d]\). Hence any asymptotically realizable \(p_0\)-free complete
signed tail family must satisfy
\(
\lambda_{\{i\},L}=\lambda_{\{i\},U}=1\) for every \(i\in[d]\).
Conversely, if the recovered weights are nonnegative and the singleton
normalization holds, then the witness parameters are valid at the tail level.
\end{proof}

\begin{corollary}[Complete-case finite-\(p_0\) realizability]
\label{cor:complete-compatibility}
For a fixed scale \(p_0\in(0,1/2)\), a complete signed tail family
\((\lambda_{J,\tau})\) admits a finite-\(p_0\) geometric witness realization if
and only if the recovered weights satisfy the conditions of
Corollary~\ref{cor:complete-tail-validity} and, in addition,
\(
p_0\sum_{(I,\sigma)\in\mathcal G_d} w_{I,\sigma}\le 1.
\)
\end{corollary}

\begin{proof}
By Corollary~\ref{cor:complete-tail-validity}, the recovered weights define a valid tail-level witness
parametrization if and only if they are nonnegative and the singleton
normalization holds. For a fixed scale \(p_0\), a finite-\(p_0\) witness
realization requires in addition that the residual central coefficient
$m_0=1-p_0\sum_{(I,\sigma)\in G_d} w_{I,\sigma}$
be nonnegative, equivalently, that the central cell mass be nonnegative. This is equivalent to
$p_0\sum_{(I,\sigma)\in G_d} w_{I,\sigma}\le 1.$
Hence the stated condition is exactly the additional finite-\(p_0\)
admissibility requirement beyond Corollary~\ref{cor:complete-tail-validity}.
\end{proof}

Thus, in the complete case, the signed tail family \(\lambda\) determines the
witness weight system \(w\) uniquely by triangular back-substitution. The next
step is to pass from witness weights to finite-\(p_0\) ternary
cell masses and exact margin conditions.

\section{Finite-\(p_0\) realization and exact margins}
\label{sec:margins}

Section~\ref{sec:tailmap} related the tail-level witness weights \(w\) to the tail coefficients
\(\lambda\). We now choose a geometric scale \(p_0\) and pass from the generator
weights to ternary cell masses \(q\) on the grid
\(
\{\LL,\MM,\UU\}^d.
\)
This separates three issues cleanly: noncentral masses are obtained from the
finite-threshold rescaling \(q_a=p_0 w_{I(a),\sigma(a)}\), the central cell
absorbs the residual probability, and exact margins are enforced by singleton
normalization.

\subsection{Ternary cells and their active patterns}\label{sec:ternary-states}

Each ternary state
\(
a=(a_1,\dots,a_d)\in\{\LL,\MM,\UU\}^d
\)
specifies, coordinatewise, whether the corresponding component lies in the
lower-tail region, in the middle region, or in the upper-tail region.

\begin{definition}[Active set and sign pattern of a ternary cell]
Let \(a\in\{\LL,\MM,\UU\}^d\). Its active set is
\[
I(a):=\{j\in[d]: a_j\neq \MM\}.
\]
If \(I(a)\neq\emptyset\), its induced sign pattern is
\[
\sigma(a):=a|_{I(a)} \in \{\LL,\UU\}^{I(a)}.
\]
\end{definition}

\begin{remark}
The central cell \(\MM^d\) is the unique ternary state with empty active set.
Every noncentral ternary cell is uniquely identified by a pair
\(
(I,\sigma)\in\mathcal G_d,
\)
namely by its active coordinates and their signs.
\end{remark}

\begin{corollary}[Counting identity]
The number of noncentral ternary cells coincides with the number of primitive
witness generators:
\[
|\{\LL,\MM,\UU\}^d\setminus\{\MM^d\}|
=
3^d-1
=
|\mathcal G_d|.
\]
\end{corollary}

\begin{proof}
The left-hand side equals \(3^d-1\) by direct counting.
The right-hand side was computed in Section~\ref{sec:construction}.
\end{proof}

\subsection{From witness weights to ternary cell masses}\label{sec:q-representation}

With the cell--generator correspondence in hand, we now attach finite-threshold masses to those cells. Fix \(p_0\in(0,1/2)\). The witness construction induces a finite realization on
the ternary partition of \([0,1]^d\), with one central cell and \(3^d-1\)
noncentral cells.

\begin{proposition}[Witness weights and ternary cell masses]
\label{prop:w-to-q}
Let \(w=(w_{I,\sigma})_{(I,\sigma)\in\mathcal G_d}\) be a witness weight system.
Then the induced ternary cell masses are given by
\[
q_a = p_0\, w_{I(a),\sigma(a)},
\qquad
a\in\{\LL,\MM,\UU\}^d\setminus\{\MM^d\},
\]
and
\[
q_{\MM^d}
=
1-p_0\sum_{(I,\sigma)\in\mathcal G_d} w_{I,\sigma}.
\]
\end{proposition}

\begin{proof}

Each noncentral generator \((I,\sigma)\) is associated with exactly one
noncentral ternary cell, namely the cell whose active set is \(I\) and whose
sign pattern on \(I\) is \(\sigma\). By construction, the geometric mass assigned to that cell is
\(p_0 w_{I,\sigma}\). Since every noncentral ternary cell corresponds to a
unique pair \((I,\sigma)\), this gives
\(
q_a = p_0\, w_{I(a),\sigma(a)}
\)
for every noncentral cell \(a\).

The remaining probability is assigned to the central cell. Hence
\[
q_{\MM^d}
=
1-\sum_{a\neq \MM^d} q_a
=
1-p_0\sum_{(I,\sigma)\in\mathcal G_d} w_{I,\sigma}.
\]
\end{proof}

\begin{corollary}
\label{cor:witness-induces-finite-p0}
The witness formulation induces a finite-\(p_0\) realization on the ternary grid,
while the tail-level parametrization \(w\mapsto\lambda\) remains \(p_0\)-free.
\end{corollary}

\begin{proof}
The finite-\(p_0\) realization follows from Proposition~\ref{prop:w-to-q},
whereas the \(p_0\)-freeness of \(w\mapsto\lambda\) was established in
Theorem~\ref{thm:w-to-lambda}.
\end{proof}

\begin{corollary}[Inverse rescaling on noncentral cells]
\label{cor:q-to-w}
For every noncentral ternary cell \(a\in\{\LL,\MM,\UU\}^d\setminus\{\MM^d\}\),
\[
w_{I(a),\sigma(a)}=\frac{q_a}{p_0}.
\]
Thus, once \(p_0\) is fixed, the noncentral witness weights and noncentral
ternary cell masses determine each other bijectively.
\end{corollary}

\begin{proof}
Immediate from Proposition~\ref{prop:w-to-q}.
\end{proof}

\begin{remark}
This relation separates the tail-level and finite-threshold descriptions.
The witness weights \(w_{I,\sigma}\) are \(p_0\)-free tail-level parameters,
whereas the cell masses \(q_a\) describe a concrete finite-\(p_0\) realization
on the ternary grid. The parameter \(p_0\) enters only through the rescaling
\(q_a=p_0 w_{I(a),\sigma(a)}\) and its inversion on noncentral cells, in Proposition~\ref{prop:w-to-q}.
\end{remark}

\subsection{Total probability and admissibility}

\begin{proposition}[Total probability identity]
\label{prop:total-probability-q}
The ternary cell masses induced by a witness weight system satisfy
\[
\sum_{a\in\{\LL,\MM,\UU\}^d} q_a = 1.
\]
Equivalently,
\[
q_{\MM^d}+\sum_{a\neq \MM^d} q_a = 1.
\]
\end{proposition}

\begin{proof}
By Proposition~\ref{prop:w-to-q},
\[
\sum_{a\neq \MM^d} q_a
=
p_0\sum_{(I,\sigma)\in\mathcal G_d} w_{I,\sigma}.
\]
Adding the central mass gives
\[
q_{\MM^d}+\sum_{a\neq \MM^d} q_a
=
\left(1-p_0\sum_{(I,\sigma)\in\mathcal G_d} w_{I,\sigma}\right)
+
p_0\sum_{(I,\sigma)\in\mathcal G_d} w_{I,\sigma}
=
1.
\]
\end{proof}

\begin{corollary}[Admissibility criterion]
\label{cor:admissibility-q}
The central cell mass is nonnegative if and only if
\[
p_0\sum_{(I,\sigma)\in\mathcal G_d} w_{I,\sigma}\le 1.
\]
Equivalently,
\(
q_{\MM^d}\ge 0.
\)
\end{corollary}

\begin{proof}
Immediate from Proposition~\ref{prop:w-to-q}.
\end{proof}

\subsection{Tail-event combinatorics on the ternary grid}\label{sec:tail-event-combinatorics}

Once the ternary masses are in place, summing them over the cells compatible
with a prescribed signed pattern recovers the corresponding signed coefficients. This is the
finite-grid counterpart of the linear map \(w\mapsto\lambda\). At this stage
it is helpful to read the ternary combinatorics literally: a target pattern
\((J,\tau)\) fixes the visible lower/upper states on the coordinates in \(J\),
while all remaining coordinates are left free. Thus the associated support set
collects exactly those ternary cells whose visible pattern on \(J\) matches
the requested signed subpattern.

For every nonempty \(J\subseteq[d]\) and every sign pattern
\(\tau\in\{\LL,\UU\}^{J}\), define the ternary support set
\[
T_{J,\tau}:=\{a\in\{\LL,\MM,\UU\}^d:\ a|_{J}=\tau\},
\]
and the associated finite-threshold tail total
\[
t_{J,\tau}^{(p_0)}(q):=\sum_{a\in T_{J,\tau}} q_a.
\]
In other words, \(t_{J,\tau}^{(p_0)}(q)\) is the total ternary mass of all
cells whose visible pattern on \(J\) matches \(\tau\).

\begin{proposition}[Tail-event totals induced by witness weights]
\label{prop:tail-event-masses}
Fix \(p_0\in(0,1/2)\), let
\(w=(w_{I,\sigma})_{(I,\sigma)\in\mathcal G_d}\) be a witness weight system,
and let \(q\) be the induced ternary cell-mass vector from
Proposition~\ref{prop:w-to-q}. Then for every nonempty \(J\subseteq[d]\) and
every \(\tau\in\{\LL,\UU\}^{J}\),
\[
t_{J,\tau}^{(p_0)}(q)
=
p_0
\sum_{I\supseteq J}
\sum_{\sigma:\,\sigma|_{J}=\tau}
w_{I,\sigma}
=
p_0\,\lambda_{J,\tau}.
\]
Equivalently,
\[
\lambda_{J,\tau}
=
\frac{t_{J,\tau}^{(p_0)}(q)}{p_0}.
\]
\end{proposition}

\begin{proof}
By Proposition~\ref{prop:w-to-q},
\[
t_{J,\tau}^{(p_0)}(q)=\sum_{a\in T_{J,\tau}} q_a.
\]
For a fixed signed target \((J,\tau)\), the coordinates in \(J\) are fixed by
\(\tau\), whereas each coordinate in \([d]\setminus J\) may be chosen freely
from \(\{\LL,\MM,\UU\}\). Hence \(T_{J,\tau}\) consists exactly of all ternary states
obtained by extending the pattern \((J,\tau)\) in this way.

Equivalently, using the one-to-one correspondence between noncentral ternary
states \(a\) and generator indices \((I(a),\sigma(a))\in G_d\), the set
\(T_{J,\tau}\) is indexed exactly by those pairs \((I,\sigma)\) such that
\(I\supseteq J\) and \(\sigma|_J=\tau\). Therefore
\[
t_{J,\tau}^{(p_0)}(q)
=\sum_{I\supseteq J}\sum_{\sigma|_J=\tau} p_0\,w_{I,\sigma}.
\]
By Theorem~\ref{thm:w-to-lambda},
\[
t_{J,\tau}^{(p_0)}(q)=p_0\,\lambda_{J,\tau},
\]
and division by \(p_0\) yields
\[
\lambda_{J,\tau}=\frac{t_{J,\tau}^{(p_0)}(q)}{p_0}.
\]
\end{proof}

\begingroup
\setlength{\intextsep}{0.2\baselineskip}
\setlength{\textfloatsep}{0.2\baselineskip}
\begin{figure}[ht]
\centering
\includegraphics[width=1\textwidth]{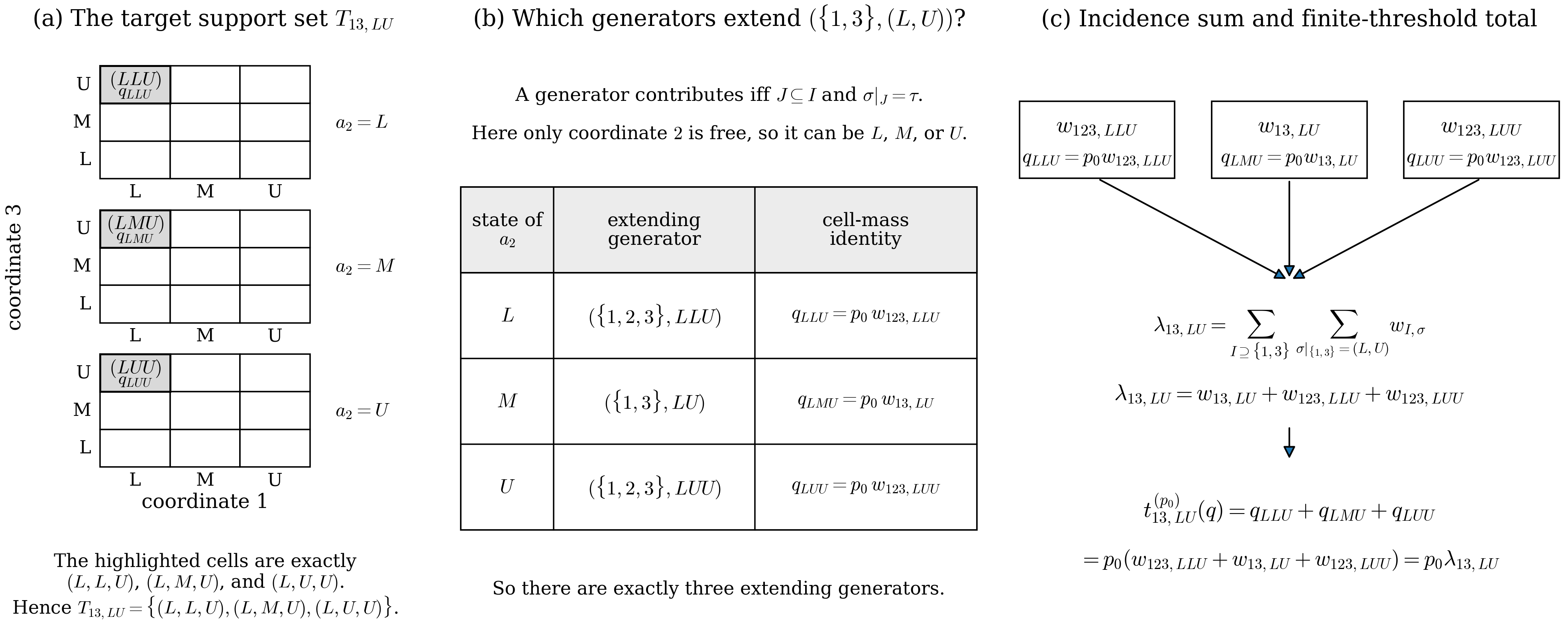}
\caption{A concrete \(d=3\) illustration of the signed extension rule and its
finite-threshold counterpart for the target pattern
\((J,\tau)=(\{1,3\},(\LL,\UU))\). Panel (a) shows the ternary support set
\(T_{13,LU}\), consisting exactly of the three compatible cells
\((\LL,\LL,\UU)\), \((\LL,\MM,\UU)\), and \((\LL,\UU,\UU)\). Panel (b) shows
the three witness generators whose signed patterns extend the same target.
Panel (c) records the corresponding incidence identity
\(\lambda_{13,LU}=w_{13,LU}+w_{123,LLU}+w_{123,LUU}\) and the matching
finite-threshold tail total
\(t_{13,LU}^{(p_0)}(q)=q_{LLU}+q_{LMU}+q_{LUU}=p_0\lambda_{13,LU}\).}
\label{fig:incidence-d3}
\end{figure}
\endgroup

Figure~\ref{fig:incidence-d3} illustrates this correspondence in dimension
\(d=3\): panel~(a) shows the compatible ternary cells, panel~(b) the
extending witness generators, and panel~(c) the matching incidence and
finite-threshold identities for
\((J,\tau)=(\{1,3\},(\LL,\UU))\).

\begin{remark}
After fixing \(p_0\), the witness weights and the ternary cell masses carry the
same information for all noncentral cells and for the induced tail totals. This
is still not a full description of arbitrary copulas with the same ternary cell
masses. In the present witness class, however, the internal generator geometry
is fixed to be the canonical ray geometry of Proposition~\ref{prop:canonical-generators}.
\end{remark}

\begin{remark}
After fixing \(p_0\), the noncentral ternary masses determine the corresponding witness weights by
\[
w_{I(a),\sigma(a)}=\frac{q_a}{p_0}.
\]
\end{remark}

\subsection{Fixed-scale invariance in the canonical ray geometry}
\label{subsec:fixed_scale_invariance}

The tail totals in Proposition~\ref{prop:tail-event-masses} describe the witness realization at the fixed
geometric scale $p_0$. For the canonical ray construction of
Proposition~\ref{prop:canonical-generators}, one can also ask what the same realized object produces at a
smaller threshold level $p\le p_0$. The answer is the key geometric fact of
the paper: within the canonical ray geometry, the realized signed tail family
does not change as $p$ moves below $p_0$.

To formulate this, fix a canonical witness realization at scale $p_0$, and let
$U=(U_1,\dots,U_d)$ denote the resulting random vector. For every nonempty
$J\subseteq[d]$, every $\tau\in\{L,U\}^J$, and every $0<p\le p_0$, define the
reflected tail coordinates
\[
D_i^{(\tau)} :=
\begin{cases}
U_i, & \tau_i=L,\\
1-U_i, & \tau_i=U,
\end{cases}
\qquad i\in J,
\]
and the corresponding finite-threshold signed coefficient
\[
\lambda^{(p)}_{J,\tau}(U)
:=
\frac1p\,
\Pr\!\bigl(D_i^{(\tau)}\le p \text{ for all } i\in J\bigr).
\]
For singleton indices this reduces to the usual lower- and upper-tail
normalization.

We first isolate the contribution of a single canonical generator.

\begin{proposition}[Single-generator fixed-scale invariance in the canonical ray geometry]
\label{prop:single-generator-fixed-scale}
Fix $(I,\sigma)\in\mathcal G_d$, a weight \(w\ge 0\), and a scale \(p_0\in(0,\frac12)\).
Consider the tail-scaled canonical generator component associated with
Proposition~\ref{prop:canonical-generators}, with total mass \(p_0w\).
Then its contribution to the finite-threshold signed coefficient is
\[
\lambda^{(p)}_{J,\tau}
=
w\,\mathbf 1\{J\subseteq I,\ \sigma|_J=\tau\}.
\]
Equivalently, a single canonical generator contributes the same amount $w$ to
every matching signed coefficient throughout the whole range $0<p\le p_0$,
and contributes $0$ to all nonmatching coefficients.
\end{proposition}

\begin{proof}
Let $U=(U_1,\dots,U_d)$ be the canonical generator from Proposition~\ref{prop:canonical-generators} with
total mass $p_0w$. All endpoint cases are understood modulo null boundary events.

First assume that $J\not\subseteq I$ or $\sigma|_J\ne\tau$. Then there exists
$i\in J$ such that either $i\notin I$ or the sign on $i$ does not match.

If $i\notin I$, then $U_i\in M_{p_0}=[p_0,1-p_0]$ almost surely. Hence
$D_i^{(\tau)}\ge p_0$ almost surely, so $D_i^{(\tau)}\le p$ is impossible for
$0<p\le p_0$, up to a null boundary event.

If $i\in I$ but $\sigma_i\ne\tau_i$, then the active coordinate lies in the
opposite tail. Thus either $U_i=1-p_0Z$ while $\tau_i=L$, or
$U_i=p_0Z$ while $\tau_i=U$. In both cases,
$D_i^{(\tau)}=1-p_0Z\in[1-p_0,1]$ almost surely, so again
$D_i^{(\tau)}\le p$ is impossible for $0<p\le p_0$.
Therefore
\[
\lambda^{(p)}_{J,\tau}=0
\]
whenever $J\not\subseteq I$ or $\sigma|_J\ne\tau$.

Now assume that $J\subseteq I$ and $\sigma|_J=\tau$. Then for every $i\in J$,
by construction of the reflected coordinates and Proposition~\ref{prop:canonical-generators},
\[
D_i^{(\tau)}=p_0Z.
\]
Hence
\[
\{D_i^{(\tau)}\le p \text{ for all } i\in J\}
=
\{p_0Z\le p\}
=
\{Z\le p/p_0\}.
\]
Since $Z\sim\mathrm{Unif}[0,1]$,
\[
\Pr(Z\le p/p_0)=p/p_0.
\]
Multiplying by the total mass $p_0w$ of the generator gives
\[
\Pr\!\bigl(D_i^{(\tau)}\le p \text{ for all } i\in J\bigr)
=
p_0w\cdot \frac{p}{p_0}
=
pw.
\]
Dividing by $p$ yields
\[
\lambda^{(p)}_{J,\tau}=w.
\]
This proves the claim.
\end{proof}

By linear superposition, the same invariance holds for the full canonical
witness. The central component contributes nothing to nonempty signed tail
events at levels $p\le p_0$, because all of its coordinates lie in $M_{p_0}$.

\begin{corollary}[Fixed-scale invariance of the full canonical witness]
\label{cor:full-witness-fixed-scale}
Fix \(p_0\in(0,1/2)\) and a nonnegative witness weight system
\(w=(w_{I,\sigma})_{(I,\sigma)\in \mathcal G_d}\) with
\[
m_0=1-p_0\sum_{(I,\sigma)\in\mathcal G_d}w_{I,\sigma}\ge 0.
\]
Let \(U\) be the corresponding canonical witness realization at scale \(p_0\).

Then for every nonempty
$J\subseteq[d]$, every $\tau\in\{L,U\}^J$, and every $0<p\le p_0$,
\[
\lambda^{(p)}_{J,\tau}(U)
=
\sum_{I\supseteq J}\ \sum_{\sigma:\,\sigma|_J=\tau} w_{I,\sigma}
=
\lambda_{J,\tau}.
\]
In particular,
\[
\lambda^{(p)}=\lambda
\qquad\text{for all }0<p\le p_0.
\]
\end{corollary}

\begin{proof}
Each noncentral canonical generator contributes according to
Proposition~\ref{prop:single-generator-fixed-scale}. Summing over all
generators therefore gives
\[
\lambda^{(p)}_{J,\tau}(U)
=
\sum_{(I,\sigma)\in\mathcal G_d}
w_{I,\sigma}\,\mathbf 1\{J\subseteq I,\ \sigma|_J=\tau\}.
\]
The central component contributes $0$ because all of its coordinates lie in
$M_{p_0}$ and therefore cannot enter a nonempty signed tail event at a level
$p\le p_0$. The displayed sum is exactly the incidence formula from
Theorem~\ref{thm:w-to-lambda}, hence
\[
\lambda^{(p)}_{J,\tau}(U)=\lambda_{J,\tau}.
\]
\end{proof}

This result explains why the canonical ray geometry is particularly natural:
once a complete signed family has been realized at scale $p_0$, the same
canonical witness automatically preserves it throughout the whole range
$0<p\le p_0$.

\subsection{Exact margins and singleton normalization}

The remaining step in this finite-\(p_0\) passage is to enforce exact one-dimensional margins.

\begin{proposition}[Exact uniform margins]
\label{prop:exact-margins}
Let \(w=(w_{I,\sigma})_{(I,\sigma)\in\mathcal G_d}\) be a witness weight system.
The associated witness mixture has exact uniform one-dimensional margins if and only if, for every
\(i\in[d]\),
\[
\sum_{\substack{(I,\sigma)\in\mathcal G_d\\ i\in I,\ \sigma_i=\LL}} w_{I,\sigma}=1,
\qquad
\sum_{\substack{(I,\sigma)\in\mathcal G_d\\ i\in I,\ \sigma_i=\UU}} w_{I,\sigma}=1.
\]
\end{proposition}

\begin{proof}

Fix \(i\in[d]\). We examine the \(i\)-th marginal on the three intervals
\(L_{p_0}\), \(U_{p_0}\), and \(M_{p_0}\). The active set \(I\) then ranges
over all generators according as \(i\in I\) or \(i\notin I\). We compare the three parts of the
\(i\)-th marginal separately: the lower tail interval \(L_{p_0}\), the upper
tail interval \(U_{p_0}\), and the middle interval \(M_{p_0}\).

On \(L_{p_0}\), only generators with \(i\in I\) and \(\sigma_i=\LL\)
contribute. Each such generator has total mass \(p_0 w_{I,\sigma}\), and its
\(i\)-th marginal is uniform on an interval of length \(p_0\); hence it
contributes the constant density \(w_{I,\sigma}\) on \(L_{p_0}\). Here the
outer sum ranges over \emph{all} active sets containing the fixed coordinate
\(i\), regardless of their cardinality. Hence the total density on 
\(L_{p_0}\)
is
\[
\sum_{\substack{(I,\sigma)\in\mathcal G_d\\ i\in I,\ \sigma_i=\LL}} w_{I,\sigma}.
\]

Similarly, on \(U_{p_0}\), the only contributors are the generators with
\(i\in I\) and \(\sigma_i=\UU\), and each such generator contributes the
constant density \(w_{I,\sigma}\) there. Thus the total density on \(U_{p_0}\)
is
\[
\sum_{\substack{(I,\sigma)\in\mathcal G_d\\ i\in I,\ \sigma_i=\UU}} w_{I,\sigma}.
\]

Therefore the \(i\)-th marginal is uniform on the two tail intervals if and
only if both displayed sums are equal to \(1\). It remains to verify the middle interval.

On \(M_{p_0}\), every generator with \(i\notin I\) contributes the constant density \(p_0 w_{I,\sigma}/(1-2p_0)\), and the central cell contributes the constant density \(q_{\MM^d}/(1-2p_0)\). Hence the total \(i\)-th marginal density on \(M_{p_0}\) is constant. Assume now that the two tail identities hold. Then the \(i\)-th marginal density is \(1\) on \(L_{p_0}\) and on \(U_{p_0}\), so these two intervals already carry masses \(p_0\) and \(p_0\). On the other hand, by the definition of the central coefficient in Definition~\ref{def:wweights}, the total mass of the witness mixture is 
\[ \Bigl(1-p_0\sum_{(I,\sigma)\in\mathcal G_d} w_{I,\sigma}\Bigr) \;+\; p_0\sum_{(I,\sigma)\in\mathcal G_d} w_{I,\sigma} =1. 
\]
Therefore the \(i\)-th marginal also has total mass \(1\). The remaining mass assigned to \(M_{p_0}\) is thus exactly \(1-2p_0\). Since \(|M_{p_0}|=1-2p_0\) and the density on \(M_{p_0}\) is constant, that density must equal \(1\). Hence the middle part of the \(i\)-th margin is also uniform.

Conversely, if the \(i\)-th margin is uniform on \([0,1]\), then its masses on
\(L_{p_0}\) and \(U_{p_0}\) must each equal \(p_0\). Dividing by \(p_0\) gives
exactly the two displayed identities. Since \(i\) was arbitrary, the result
follows for all coordinates. Hence the displayed conditions are both
necessary and sufficient for exact uniform one-dimensional margins.
\end{proof}

\begin{corollary}[Singleton normalization]
\label{cor:singleton-normalization}
For a witness weight system, the exact margin conditions are equivalent to
\[
\lambda_{\{i\},\LL}=1,
\qquad
\lambda_{\{i\},\UU}=1,
\qquad i=1,\dots,d.
\]
\end{corollary}

\begin{proof}
For \(J=\{i\}\) and \(\tau=\LL\), Theorem~\ref{thm:w-to-lambda} gives
\[
\lambda_{\{i\},\LL}
=
\sum_{(I,\sigma)\in G_d:\, i\in I,\ \sigma_i=\LL} w_{I,\sigma}.
\]
Similarly, for \(\tau=\UU\),
\[
\lambda_{\{i\},\UU}
=
\sum_{(I,\sigma)\in G_d:\, i\in I,\ \sigma_i=\UU} w_{I,\sigma}.
\]
Comparing these identities with Proposition~\ref{prop:exact-margins} yields the claim.
\end{proof}

Once the tail coefficients and exact margins are expressed in the witness
parametrization, it is natural to compare them with the finite-\(p_0\)
ternary cell masses used for realization and simulation.

\begin{remark}[Tail-level validity versus finite-$p_0$ admissibility]
Tail-level validity is encoded by the nonnegative witness weights together
with the singleton normalization and the linear relations $\lambda=Aw$.
For the canonical ray realization, Corollary~\ref{cor:full-witness-fixed-scale}
shows that the same complete signed family is then preserved throughout the
whole threshold range $0<p\le p_0$. A finite-$p_0$ realization requires one
additional geometric condition, namely nonnegativity of the residual central
mass after the passage $q_a=p_0 w_{I(a),\sigma(a)}$.
\end{remark}

\newpage

\section{Partial, noisy, and inconsistent targets via LP}
\label{sec:lp}

Section~\ref{sec:tailmap} provides the tail-level incidence/inversion layer, while Section~\ref{sec:margins} provides the fixed-scale rescaling to ternary masses. We now use the same witness parameterization for partial, noisy, and inconsistent targets.

\subsection{Complete case: inversion rather than optimization}

Suppose that the complete signed tail family
\[
\bigl(\lambda_{J,\tau}\bigr)_{\emptyset\neq J\subseteq[d],\ \tau\in\{\LL,\UU\}^{J}}
\]
is given.

By Proposition~\ref{prop:triangular-inversion}, the witness weights are then
recovered recursively by
\[
w_{I,\sigma}
=
\lambda_{I,\sigma}
-
\sum_{K\supsetneq I}
\ \sum_{\rho:\,\rho|_I=\sigma}
w_{K,\rho},
\qquad
(I,\sigma)\in\mathcal G_d,
\]
starting from active sets of size \(d\) and proceeding down to active sets of
size \(1\). In this recursion, \(K\supsetneq I\) denotes a strict active-set
superset, while \(\rho|_I=\sigma\) means that the sign pattern on the larger
set \(K\) agrees with \(\sigma\) when restricted back to the coordinates of
\(I\).

This recursion is already the complete direct-inversion procedure: compute the
weights by descending active-set size, then check nonnegativity and singleton
normalization as in Corollary~\ref{cor:complete-tail-validity}. The geometric
realization scale \(p_0\) plays no role in the recovery itself. Once the recovered
system is tail-level valid, finite-\(p_0\) admissibility is checked separately
through the residual central-mass condition; see
Corollary~\ref{cor:complete-compatibility} and Section~\ref{sec:asymptotic}.

\begin{remark}
For complete specifications, witness recovery is therefore an explicit linear
elimination problem rather than a genuine optimization problem. Linear
programming becomes necessary only when the target information is incomplete,
noisy, or internally inconsistent.
\end{remark}

\subsection{Feasibility LP for partial specifications}

We now consider the more general situation in which only part of the complete
signed tail family is prescribed.

Let
\[
\mathcal S
\subseteq
\{(J,\tau): \emptyset\neq J\subseteq[d],\ \tau\in\{\LL,\UU\}^{J}\}
\]
be the set of specified signed tail coefficients, and let
\[
\lambda^{\mathrm{target}}_{J,\tau},
\qquad
(J,\tau)\in\mathcal S,
\]
denote the prescribed values. Thus \(\mathcal S\) is an arbitrary subset of the complete signed tail family
encoding the partial specification under consideration.

The unknowns are the witness weights \(w_{I,\sigma}\ge 0\), \((I,\sigma)\in G_d\).
For a prescribed partial specification at scale \(p_0\), the witness constraints take the form

\begin{align}
w_{I,\sigma} &\ge 0,
&& (I,\sigma)\in\mathcal G_d,
\label{eq:lp-nonneg}
\\
\sum_{I\supseteq J}
\ \sum_{\sigma:\,\sigma|_J=\tau}
w_{I,\sigma}
&=
\lambda^{\mathrm{target}}_{J,\tau},
&& (J,\tau)\in\mathcal S,
\label{eq:lp-tail}
\\
\sum_{\substack{(I,\sigma)\in\mathcal G_d\\ i\in I,\ \sigma_i=\LL}} w_{I,\sigma}
&=1,
&& i=1,\dots,d,
\label{eq:lp-margin-L}
\\
\sum_{\substack{(I,\sigma)\in\mathcal G_d\\ i\in I,\ \sigma_i=\UU}} w_{I,\sigma}
&=1,
&& i=1,\dots,d,
\label{eq:lp-margin-U}
\\
\sum_{(I,\sigma)\in\mathcal G_d} w_{I,\sigma}
&\le \frac{1}{p_0}.
\label{eq:lp-admiss}
\end{align}

\begin{proposition}[LP feasibility formulation]
\label{prop:lp-feasibility}
A partial signed specification \(\bigl(\lambda^{\mathrm{target}}_{J,\tau}\bigr)_{(J,\tau)\in\mathcal S}\)
is realizable within the witness class at scale \(p_0\) if and only if there exists
a witness weight system \(w=(w_{I,\sigma})_{(I,\sigma)\in G_d}\) satisfying
\eqref{eq:lp-nonneg}--\eqref{eq:lp-admiss}.
\end{proposition}

\begin{proof}
The statement is a reformulation of realizability within the witness class in
terms of the witness weights. Indeed, if a partial signed specification is
realizable within the witness class at scale \(p_0\), then by definition there
exists a witness weight system whose coefficients match the prescribed tail
equations, satisfy nonnegativity, enforce the exact margin conditions, and
meet the admissibility bound. Hence \eqref{eq:lp-nonneg}--\eqref{eq:lp-admiss}
is feasible.

Conversely, let \(w=(w_{I,\sigma})_{(I,\sigma)\in\mathcal G_d}\) be any feasible
solution of \eqref{eq:lp-nonneg}--\eqref{eq:lp-admiss}. By
\eqref{eq:lp-tail} and Theorem~\ref{thm:w-to-lambda}, the specified signed tail coefficients are
matched. By \eqref{eq:lp-margin-L}, \eqref{eq:lp-margin-U}, and
Proposition~\ref{prop:exact-margins}, the exact margin conditions hold. By Proposition~\ref{prop:w-to-q} and
\eqref{eq:lp-admiss}, the induced central mass
\(q_{M^d}=1-p_0\sum_{(I,\sigma)\in\mathcal G_d} w_{I,\sigma}\)
is nonnegative. Thus the feasible weight system defines an admissible
finite-\(p_0\) witness realization matching the prescribed partial
specification.
\end{proof}

\begin{remark}
If the geometric scale \(p_0\) is not fixed in advance, then the admissibility
constraint \eqref{eq:lp-admiss} may be omitted at the tail level and imposed
only later, when a concrete finite-scale realization is selected.
\end{remark}

\subsection{Noisy repair by weighted-\(\ell^1\)}

In practical calibration, the prescribed coefficients may be estimated from data,
elicited from experts, or assembled from heterogeneous sources. Such targets may
fail to be exactly realizable within the witness class. In this case, one may
seek a \emph{closest compatible} witness profile.

A convenient linear formulation uses positive and negative slack variables
\[
r_{J,\tau}^+\ge 0,
\qquad
r_{J,\tau}^-\ge 0,
\qquad
(J,\tau)\in\mathcal S,
\]
and replaces the exact tail equations \eqref{eq:lp-tail} by
\begin{equation}
\sum_{I\supseteq J}
\ \sum_{\sigma:\,\sigma|_J=\tau}
w_{I,\sigma}
-
\lambda^{\mathrm{target}}_{J,\tau}
=
r_{J,\tau}^+ - r_{J,\tau}^-,
\qquad
(J,\tau)\in\mathcal S.
\label{eq:lp-slack}
\end{equation}
The quantity
\(
r_{J,\tau}^+ + r_{J,\tau}^-
\)
then measures the absolute deviation from the target coefficient
\(\lambda^{\mathrm{target}}_{J,\tau}\).

Given strictly positive calibration weights \(\omega_{J,\tau}>0\) for all
\((J,\tau)\in\mathcal S\), one may minimize the weighted \(\ell^1\)-misfit
\begin{equation}
\min
\sum_{(J,\tau)\in\mathcal S}
\omega_{J,\tau}\bigl(r_{J,\tau}^+ + r_{J,\tau}^-\bigr)
\label{eq:lp-L1-objective}
\end{equation}
subject to
\eqref{eq:lp-nonneg}, \eqref{eq:lp-margin-L}, \eqref{eq:lp-margin-U},
\eqref{eq:lp-admiss}, and \eqref{eq:lp-slack}.

\begin{proposition}[Approximate recovery by linear optimization]
\label{prop:lp-approx}
Assume strictly positive calibration weights \(\omega_{J,\tau}>0\) for all
\((J,\tau)\in\mathcal S\). Then problem \eqref{eq:lp-L1-objective} with
constraints \eqref{eq:lp-nonneg}, \eqref{eq:lp-margin-L},
\eqref{eq:lp-margin-U}, \eqref{eq:lp-admiss}, and \eqref{eq:lp-slack} is a
linear program. Its optimal value is zero if and only if the target
specification is exactly realizable within the witness class at scale \(p_0\).
\end{proposition}

\begin{proof}
The objective \eqref{eq:lp-L1-objective} is linear in the slack variables, and
all constraints are linear in the optimization variables
\(
\bigl(w_{I,\sigma}, r_{J,\tau}^+, r_{J,\tau}^-\bigr).
\)
Hence the problem is a linear program.

If the optimal value is zero, then
\(
\sum_{(J,\tau)\in\mathcal S}
\omega_{J,\tau}\bigl(r_{J,\tau}^+ + r_{J,\tau}^-\bigr)=0.
\)
Since each \(\omega_{J,\tau}\) is strictly positive and each slack variable is
nonnegative, it follows that
\(r_{J,\tau}^+=r_{J,\tau}^-=0\) for every \((J,\tau)\in\mathcal S\).
Therefore the exact tail equations hold, together with the margin and
admissibility constraints, so the target specification is exactly realizable
within the witness class at scale \(p_0\).

Conversely, if the target specification is exactly realizable, one may choose a
feasible witness weight system satisfying the exact tail equations and set all
slack variables equal to zero. This yields objective value zero.
\end{proof}

\begin{remark}
The weighted \(\ell^1\)-formulation is attractive because it preserves linearity
and allows the user to emphasize coefficients of greater practical importance.
Strictly positive calibration weights ensure that zero objective value is
equivalent to exact realizability, while larger weights may be assigned to
coefficients regarded as more important in a given application.
\end{remark}

\subsection{Secondary selection criteria}

Whenever the feasible set is nonempty, one may impose additional linear
selection criteria reflecting modelling preferences or practical priorities.

A simple example is the minimization of total noncentral witness mass:
\[
\min \sum_{(I,\sigma)\in\mathcal G_d} w_{I,\sigma}.
\]
This tends to favor witness realizations with larger central mass
\[
q_{\MM^d}
=
1-p_0\sum_{(I,\sigma)\in\mathcal G_d} w_{I,\sigma}.
\]

More generally, given nonnegative costs \(c_{I,\sigma}\), one may solve
\[
\min \sum_{(I,\sigma)\in\mathcal G_d} c_{I,\sigma} w_{I,\sigma}
\]
subject to the same feasibility constraints. This permits, for instance,
preferential penalization of higher-order generators or of sign configurations
considered less plausible in a given application.

\begin{remark}
Such secondary objectives remain within linear programming as long as the
criterion is linear in the witness weights. This makes it possible to encode
expert judgement directly at the level of interpretable joint-behaviour
generators.
\end{remark}

\subsection{Scope of the LP layer}

The LP layer begins only after the exact complete-case inversion has exhausted
what can be done algebraically. Its basic role is therefore not to replace the
triangular witness inversion, but to extend the same witness parametrization to
partial, noisy, and inconsistent targets.

A particularly important nonlinear second-stage selector is maximum entropy on
the ternary cell masses q; compare the general convex-optimization viewpoint in
\cite{BoydVandenberghe2004} and the maximum-entropy perspective in
\cite{Jaynes2003}. Since \(q\) depends linearly on \(w\) and the Shannon
entropy \(-\sum_a q_a\log q_a\) is strictly concave on the simplex, the
entropy-maximizing feasible ternary profile, when it exists, is unique. After
fixing the canonical geometric generators of
Proposition~\ref{prop:canonical-generators}, this yields a distinguished finite-
\(p_0\) witness copula among all LP-feasible realizations. A simple alternative
is the order-\(2\) Rényi entropy~\citep{Renyi1961}, whose maximization is
equivalent to minimizing \(\sum_a q_a^2\) under the same linear constraints,
hence to a convex quadratic-programming problem. We do not make this nonlinear
layer part of the core solver, but the resulting copula selection principle and
the Shannon-based uniqueness consequence are important to keep in view.

\section{Simulation, benchmark, and computational validation}
\label{sec:simulation}

This section illustrates and validates the constructive line of the paper. Once a
feasible weight system $w$ has been obtained, it determines both a finite-$p_0$
ternary realization through the cell masses $q_a$ and a canonical simulation
mechanism for the corresponding witness copula. We then use a compact
five-dimensional benchmark to check the full pipeline
$\lambda \to w \to q \to$ simulation.

\subsection{Canonical sampling and numerical diagnostics}

Fix a feasible witness weight system $w$ and a scale $p_0\in(0,1/2)$. Define
\[
\pi_{I,\sigma}:=p_0 w_{I,\sigma},
\qquad
\pi_0:=1-p_0\sum_{(I,\sigma)\in\mathcal G_d} w_{I,\sigma}.
\]
Then $\pi_0$ and the $\pi_{I,\sigma}$ form a probability distribution over the
central component and the noncentral generators. Using the canonical geometric
generators from Proposition~\ref{prop:canonical-generators}, exact simulation is
therefore immediate: draw the component label with these probabilities and then
sample from the corresponding canonical generator or from the central
component.

\begin{proposition}[Canonical witness sampler]
\label{prop:witness-sampler-correct}
The two-stage mixture sampler described above generates exactly the witness
copula determined by $w$ and $p_0$.
\end{proposition}

\begin{proof}
By construction, the witness copula is precisely the mixture of the central
component and the canonical generators with weights $\pi_0$ and $\pi_{I,\sigma}$.
Sampling the label first and then sampling conditionally from the selected
component reproduces this mixture law exactly.
\end{proof}

For numerical checks, let $U^{(1)},\dots,U^{(M)}$ be simulated samples. For any
nonempty $J\subseteq[d]$, sign pattern $\tau\in\{\LL,\UU\}^{J}$, and threshold
$p\in(0,1/2)$, we use the empirical signed tail estimator
\[
\widehat{\lambda}_{J,\tau}^{(M)}(p)
:=
\frac{1}{Mp}
\sum_{m=1}^M
\prod_{j\in J}
\mathbf 1\!\left\{U_j^{(m)}\in T_{\tau_j}(p)\right\},
\]
where $T_{\LL}(p)=[0,p]$ and $T_{\UU}(p)=[1-p,1]$. In practice it is enough to
check consistency across three layers: recovered witness weights $w$, induced
ternary masses $q$, and empirical signed tail behaviour under simulation.

\subsection{Five-dimensional benchmark}
\label{sec:benchmark}

We now use a compact five-dimensional benchmark that plays two roles. First, it
provides a nontrivial complete signed specification for which the witness
weights can be recovered analytically. Second, it serves as a solver benchmark,
since the same target family can be treated both by triangular inversion and by
the LP formulations of Section~\ref{sec:lp}. The construction extends the
sparse upper-tail benchmark first presented in~\cite{MB11} and later used
in~\citet*{EHW2016}.

\subsubsection{Benchmark specification}

Fix $d=5$ and let $\alpha\ge 0$ be a tuning parameter.

\paragraph{Singleton coefficients.}
For each coordinate \(i\in[5]\), \(\lambda_{\{i\},\LL}=1\) and
\(\lambda_{\{i\},\UU}=1\).

\paragraph{Pair coefficients.}
For signs \(s,t\in\{\LL,\UU\}\), write
\[
(\Lambda^{st})_{ij}:=\lambda_{\{i,j\},(s,t)}, \qquad 1\le i<j\le 5,
\]
and display only the upper-triangular entries, marking the diagonal and all
lower-triangular positions by dots because they are not part of the displayed
data. The singleton normalization is kept separately. We set
\[
\Lambda^{\UU\UU}=\Lambda^{\LL\LL}=
\begin{pmatrix}
\cdot & 0 & 0 & 0 & \alpha\\
\cdot & \cdot & 0 & 0 & \alpha\\
\cdot & \cdot & \cdot & 0 & \alpha\\
\cdot & \cdot & \cdot & \cdot & \alpha\\
\cdot & \cdot & \cdot & \cdot & \cdot
\end{pmatrix},
\qquad
\Lambda^{\UU\LL}=\Lambda^{\LL\UU}=
\begin{pmatrix}
\cdot & 1 & 0 & 0 & \alpha\\
\cdot & \cdot & 0 & 0 & \alpha\\
\cdot & \cdot & \cdot & 1 & \alpha\\
\cdot & \cdot & \cdot & \cdot & \alpha\\
\cdot & \cdot & \cdot & \cdot & \cdot
\end{pmatrix}.
\]
Thus the only nonzero signed pair coefficients among the first four
variables are the opposite-sign pairs \(\{1,2\}\) and \(\{3,4\}\), both
equal to \(1\), while every pair involving coordinate \(5\) carries the common
value \(\alpha\) in all four sign configurations.

\paragraph{Third-order coefficients.}
For \(t\in\{\LL,\UU\}\), the only nonzero signed triple coefficients are
\[
\lambda_{\{1,2,5\},(\UU,\LL,t)}
=\lambda_{\{1,2,5\},(\LL,\UU,t)}
=\lambda_{\{3,4,5\},(\UU,\LL,t)}
=\lambda_{\{3,4,5\},(\LL,\UU,t)}
=\alpha.
\]
All other signed triple coefficients are zero, and all signed coefficients of
orders \(4\) and \(5\) are set to zero.

\begin{remark}
The benchmark is complete in the sense of Section~\ref{sec:tailmap}: all
singleton coefficients are fixed, all signed coefficients of orders $2$ and $3$
are specified, and all higher-order coefficients are prescribed to be zero.
\end{remark}

\subsubsection{Exact recovery of the witness weights}

\begin{proposition}[Explicit witness solution for the benchmark]
\label{prop:benchmark-explicit-w}
For the five-dimensional benchmark above, triangular inversion yields the
following nonzero witness weights:
\[
\begin{array}{@{}l@{\;}c@{\;}l@{\;}c@{\;}l@{\;}c@{\;}l@{\;}c@{\;}l@{}}
w_{\{1,2,5\},(\UU,\LL,\UU)} & = & w_{\{1,2,5\},(\UU,\LL,\LL)} & = & w_{\{1,2,5\},(\LL,\UU,\UU)} & = & w_{\{1,2,5\},(\LL,\UU,\LL)} & = & \alpha,
\\
w_{\{3,4,5\},(\UU,\LL,\UU)} & = & w_{\{3,4,5\},(\UU,\LL,\LL)} & = & w_{\{3,4,5\},(\LL,\UU,\UU)} & = & w_{\{3,4,5\},(\LL,\UU,\LL)} & = & \alpha,
\\
w_{\{1,2\},(\UU,\LL)} & = & w_{\{1,2\},(\LL,\UU)} & = & w_{\{3,4\},(\UU,\LL)} & = & w_{\{3,4\},(\LL,\UU)} & = & 1-2\alpha,
\\
w_{\{5\},\UU} & = & w_{\{5\},\LL} &   &   &   &   & = & 1-4\alpha.
\end{array}
\]
All remaining witness weights are zero.
\end{proposition}

\begin{proof}
All coefficients of orders $4$ and $5$ vanish, so the corresponding weights
vanish as well. At order $3$, there are no nonzero supersets left, hence each
nonzero triple coefficient becomes the corresponding triple weight, giving the
eight weights equal to $\alpha$ above. At order $2$, each nonzero opposite-sign
pair on $\{1,2\}$ or $\{3,4\}$ loses exactly two triple contributions of size
$\alpha$, so its residual pair weight is $1-2\alpha$; all other pair weights
remain zero. Finally, the singleton coefficients for variables $1,2,3,4$ are
already exhausted by pair and triple contributions, while coordinate $5$
retains two singleton weights $1-4\alpha$, one in each sign. This is exactly
the displayed solution.
\end{proof}

\begin{corollary}[Tail-level compatibility interval]
\label{cor:benchmark-alpha-interval}
The benchmark determines a valid tail-level witness parametrization if and only
if
\(
\alpha\in[0,1/4].
\)
\end{corollary}

\begin{proof}
By Proposition~\ref{prop:benchmark-explicit-w}, nonnegativity of the recovered
weights is equivalent to $1-2\alpha\ge 0$ and $1-4\alpha\ge 0$, hence to
$\alpha\in[0,1/4]$.
\end{proof}

\subsubsection{Finite-$p_0$ realization of the benchmark}

For a fixed \(p_0\in(0,1/2)\), the benchmark induces noncentral cell masses
\(q_a=p_0\,w_{I(a),\sigma(a)}\) for \(a\neq \MM^5\), and central mass
\(q_{\MM^5}=1-p_0\sum_{(I,\sigma)\in\mathcal G_5} w_{I,\sigma}\).
Since \(\sum_{(I,\sigma)\in\mathcal G_5} w_{I,\sigma}
=8\alpha+4(1-2\alpha)+2(1-4\alpha)=6-8\alpha\), we obtain
\(q_{\MM^5}=1-p_0(6-8\alpha)\).

\begin{corollary}[Finite-$p_0$ admissibility for the benchmark]
\label{cor:benchmark-p0-admiss}
For a fixed $\alpha\in[0,1/4]$, the benchmark admits a finite-$p_0$ ternary
realization whenever
\(
p_0\le \frac{1}{6-8\alpha}.
\)
\end{corollary}

\begin{proof}
This is exactly the condition $q_{\MM^5}\ge 0$.
\end{proof}

\begin{remark}
The benchmark separates the two levels cleanly: the tail-level compatibility
interval $\alpha\in[0,1/4]$ is $p_0$-free, while the finite realization level
re-enters only through the nonnegativity of the central cell mass.
\end{remark}

The same formula also makes the distinction between tail-level validity and
finite-$p_0$ admissibility transparent. For instance, at the larger scale
$p_0=0.20$,
\(
q_{\MM^5}=1-0.2(6-8\alpha)=-0.2+1.6\alpha,
\)
so the tail-level feasible interval $\alpha\in[0,1/4]$ shrinks at the finite
realization level to $\alpha\in[1/8,1/4]$. Thus $\alpha=0.10$ is tail-level
compatible but not admissible at $p_0=0.20$, whereas $\alpha=0.20$ is
admissible at both levels.

\subsection{Computational validation}
\label{sec:experiments}

The benchmark provides an exact reference boundary, so the computational checks are compact.

Table~\ref{tab:benchmark-complete} reports a small grid of values for the
same five-dimensional benchmark at the fixed realization scale $p_0=0.10$,
now with Monte Carlo diagnostics evaluated at the two thresholds
$p=p_0$ and $p=p_0/2=0.05$. Direct inversion and LP feasibility agree
throughout. Monte Carlo entries are reported only for the complete-case
feasible rows and summarize, over repeated canonical-sampling runs, the mean
and standard deviation of the per-run maximum absolute deviation across the
nonzero benchmark target tail coefficients of orders $2$ and $3$; singleton
coefficients are checked separately as exact-margin sanity tests. The
additional $p_0/2$ columns test the same benchmark against the fixed-scale
invariance result from Section~\ref{subsec:fixed_scale_invariance}, without introducing a
separate test family. As expected, the smaller threshold is slightly noisier,
but the error levels remain of the same order. For $\alpha=0.26$, the central
mass is still positive at this scale, but the target is already tail-level
infeasible because $w_{\{5\},\UU}=w_{\{5\},\LL}=1-4\alpha<0$.

\begin{table}[H]
\centering
\caption{Complete-case validation for the five-dimensional benchmark at
$p_0=0.10$. Monte Carlo summaries are based on $R=20$ independent runs,
each with $M=5\times 10^5$ canonical samples. For each run, the reported
Monte Carlo error is the maximum absolute deviation over the nonzero
benchmark target tail coefficients of orders $2$ and $3$ only; singleton
coefficients are checked separately through exact-margin sanity tests.
The same benchmark is evaluated at the two thresholds $p=p_0$ and
$p=p_0/2=0.05$.}
\label{tab:benchmark-complete}
\begin{tabular}{@{}ccccccccc@{}}
\toprule
$\alpha$
& \shortstack[c]{minimum\\ recovered $w$}
& \shortstack[c]{direct\\ inversion}
& \shortstack[c]{LP\\ feasibility}
& $q_{\MM^5}$
& \shortstack[c]{MC mean\\ max error\\ @ $p_0$}
& \shortstack[c]{MC sd\\ max error\\ @ $p_0$}
& \shortstack[c]{MC mean\\ max error\\ @ $p_0/2$}
& \shortstack[c]{MC sd\\ max error\\ @ $p_0/2$} \\
\midrule
$0.00$ & $0.00$  & feasible   & feasible   & $0.400$ & 0.0057 & 0.0027 & 0.0091 & 0.0028 \\
$0.10$ & $0.00$  & feasible   & feasible   & $0.480$ & 0.0062 & 0.0025 & 0.0078 & 0.0031 \\
$0.20$ & $0.00$  & feasible   & feasible   & $0.560$ & 0.0065 & 0.0023 & 0.0099 & 0.0027 \\
$0.24$ & $0.00$  & feasible   & feasible   & $0.592$ & 0.0058 & 0.0021 & 0.0088 & 0.0038 \\
$0.25$ & $0.00$  & feasible   & feasible   & $0.600$ & 0.0071 & 0.0021 & 0.0084 & 0.0032 \\
$0.26$ & $-0.04$ & infeasible & infeasible & $0.608$ & --     & --     & --     & --     \\
\bottomrule
\end{tabular}
\end{table}

Beyond the complete exact case, two short LP experiments show how the same
witness parametrization handles partial and inconsistent targets. For
$\alpha=0.20$, prescribing only singleton coefficients and the four
opposite-sign pair coefficients admits much sparser completions than the
benchmark witness system, and minimizing the total noncentral weight selects such a sparse completion automatically. For the inconsistent target $\alpha=0.26$, the weighted
$\ell^1$ repair LP from Proposition~\ref{prop:lp-approx} returns a feasible
repaired witness system with exact singleton normalization and nonnegative
central mass at $p_0=0.10$. These experiments complete the computational
picture: exact validation by inversion and LP agreement, simulation as a
numerical sanity check, and LP-based completion or repair beyond the complete
exact regime. In practical calibration, $p_0$ may then be chosen below the
admissibility bound from Section~\ref{sec:asymptotic}; here it is kept fixed
only to separate the computational layer from the later asymptotic analysis.

This concludes the operational line of the paper. It remains only to record
formally that the same finite-$p_0$ construction defines a copula.

\section{Formal measure-theoretic realization of witness copulas}
\label{sec:formal-shell}

We now record the measure-theoretic shell showing that the finite-\(p_0\)
witness construction defines a copula. No new incidence algebra is used here.

\subsection{Normalized generators and the central component}

The witness construction is formulated at the level of copula measures.

\begin{definition}[Normalized generator]
\label{def:normalized-generator}
For \((I,\sigma)\in\mathcal G_d\), a normalized generator at scale \(p_0\) is a
probability measure \(\mu_{I,\sigma}^{(p_0)}\) on \([0,1]^d\) such that:
\begin{enumerate}
    \item \(\mu_{I,\sigma}^{(p_0)}\) is supported on the signed cell
    \(B_{I,\sigma}^{(p_0)}\);
    \item for each coordinate \(i\in[d]\), the \(i\)-th marginal of
    \(\mu_{I,\sigma}^{(p_0)}\) is the normalized uniform distribution on
    \(A_i^{(I,\sigma)}\), that is, it has density
    \[
    \frac{1}{|A_i^{(I,\sigma)}|}\,\mathbf 1_{A_i^{(I,\sigma)}}
    \]
    with respect to Lebesgue measure on \([0,1]\).
\end{enumerate}
\end{definition}

\begin{definition}[Central component]
\label{def:central-component}
A central component at scale \(p_0\) is a probability measure
\[
\mu_0^{(p_0)}
\]
on \([0,1]^d\) such that:
\begin{enumerate}
    \item \(\mu_0^{(p_0)}\) is supported on the central cell \(B_0^{(p_0)}\);
    \item each one-dimensional marginal of \(\mu_0^{(p_0)}\) is the normalized
    uniform distribution on \(M_{p_0}\), that is, the probability measure with
    density
    \[
    \frac{1}{|M_{p_0}|}\,\mathbf 1_{M_{p_0}}
    \]
    with respect to Lebesgue measure on \([0,1]\).
\end{enumerate}
\end{definition}

\begin{remark}
Definitions~\ref{def:normalized-generator} and~\ref{def:central-component}
isolate the support and marginal facts needed for the measure-theoretic check.
In the witness copulas considered in this paper, however, these measures are
fixed to be the laws of the canonical shared-factor construction in
Proposition~\ref{prop:canonical-generators} and of the corresponding central
component. Thus the freedom of the construction lies in \(w\) and \(p_0\), not
in choosing alternative in-cell geometries.
\end{remark}

\subsection{Tail-scaled generators and witness measure}

\begin{definition}[Tail-scaled generator component]
\label{def:tail-scaled-generator}
For \((I,\sigma)\in\mathcal G_d\), define the associated tail-scaled generator
component by
\[
\widetilde\mu_{I,\sigma}^{(p_0)}
:=
p_0\,\mu_{I,\sigma}^{(p_0)}.
\]
This is a finite nonnegative measure on \([0,1]^d\) of total mass \(p_0\).
\end{definition}

\begin{definition}[Witness copula measure]
\label{def:witness-copula}
Fix the canonical family of normalized generators induced by
Proposition~\ref{prop:canonical-generators}, together with the corresponding
canonical central component.

Given a witness weight system
\(w\) with \(m_0\ge 0\), define the witness measure \(\mu^{(w,p_0)}\) on
\([0,1]^d\) by
\[
\mu^{(w,p_0)}
:=
m_0\,\mu_0^{(p_0)}
+
\sum_{(I,\sigma)\in\mathcal G_d}
w_{I,\sigma}\,\widetilde\mu_{I,\sigma}^{(p_0)}.
\]

If the exact margin conditions of
Proposition~\ref{prop:exact-margins} hold, then \(\mu^{(w,p_0)}\) has
uniform one-dimensional margins on \([0,1]\), and its associated
distribution function
\[
C^{(w,p_0)}(u_1,\dots,u_d)
:=
\mu^{(w,p_0)}
\bigl([0,u_1]\times\cdots\times[0,u_d]\bigr)
\]
is called the geometric witness copula.
\end{definition}

\begin{remark}
The role of \(p_0\) is geometric rather than structural. It determines the
finite ternary realization and the central mass, but the signed tail map
developed in the next section will be shown to be independent of \(p_0\).
\end{remark}

\subsection{Formal copula formula and legal witness realization}

\begin{proposition}[Formal witness copula realization]
\label{prop:formal-witness-copula}

Fix \(p_0\in(0,1/2)\), take the canonical normalized generators and central
component induced by Proposition~\ref{prop:canonical-generators}, and let \(w\)
be a nonnegative witness weight system. If the exact
margin conditions of Proposition~\ref{prop:exact-margins} hold and the central
coefficient
\[
 m_0 = 1-p_0\sum_{(I,\sigma)\in\mathcal G_d} w_{I,\sigma}
\]
is nonnegative, then
\[
\mu^{(w,p_0)}
:=
m_0\,\mu_0^{(p_0)}
+
\sum_{(I,\sigma)\in\mathcal G_d}
 w_{I,\sigma}\,\widetilde\mu_{I,\sigma}^{(p_0)}
\]
is a probability measure on \([0,1]^d\) with uniform one-dimensional margins.
Consequently,
\[
C^{(w,p_0)}(u_1,\dots,u_d)
:=
\mu^{(w,p_0)}\bigl([0,u_1]\times\cdots\times[0,u_d]\bigr)
\]
is a copula.
\end{proposition}

\begin{proof}
The existence of such normalized generators and of the central component is
guaranteed by Proposition~\ref{prop:canonical-generators}. Every coefficient in the defining mixture is
nonnegative, and the total mass is
\(m_0 + p_0\sum w_{I,\sigma}=1\). Proposition~\ref{prop:exact-margins}
provides the exact one-dimensional marginal equalities, so the mixture has
uniform margins on \([0,1]\). The associated distribution function is therefore
a copula.
\end{proof}

This formal realization also shows that geometric witness copulas are not
just a single constructive device but a genuine copula family, and the first
structural closure property is stability under coordinate marginalization. 

\begin{corollary}[Closure under marginalization]
\label{cor:witness-marginalization}
Let \(C^{(d)}=C(w,p_0)\) be a geometric witness copula in dimension \(d\),
and let \(S\subseteq[d]\) be nonempty. Then the \(S\)-marginal of \(C^{(d)}\)
is again a geometric witness copula at the same scale \(p_0\). Moreover, for
every nonempty \(J\subseteq S\) and every \(\tau\in\{L,U\}^J\),
\[
\lambda^{(S)}_{J,\tau}=\lambda^{(d)}_{J,\tau}.
\]
Thus coordinate reduction preserves the signed tail coefficients on the
remaining coordinates.
\end{corollary}

\begin{proof}
Project the witness measure \(\mu(w,p_0)\) from \([0,1]^d\) onto the
coordinates in \(S\). For every nonempty \(K\subseteq S\) and every
\(\rho\in\{L,U\}^K\), define the projected witness weights
\[
\widetilde w_{K,\rho}
:=
\sum_{\substack{(I,\sigma)\in G_d:\\ I\cap S = K,\ \sigma|_K=\rho}}
w_{I,\sigma}.
\]
Generators with \(I\cap S\neq\varnothing\) project to lower-dimensional
normalized generators indexed by \((K,\rho)\), while generators with
\(I\cap S=\varnothing\) project into the lower-dimensional central cell and
are absorbed into the projected central component, so admissibility is
preserved. Hence the projected measure is again a witness measure at the same
scale \(p_0\), with witness system \(\widetilde w\).

Now let \(\varnothing\neq J\subseteq S\) and \(\tau\in\{L,U\}^J\). By the
incidence formula of Theorem~1,
\[
\lambda^{(S)}_{J,\tau}
=
\sum_{K\supseteq J}\ \sum_{\rho|_J=\tau}\widetilde w_{K,\rho}.
\]
Substituting the definition of \(\widetilde w_{K,\rho}\) and using that
\(K=I\cap S\) and \(\rho=\sigma|_K\), we see that, for \(J\subseteq S\), the
conditions \(K\supseteq J\) and \(\rho|_J=\tau\) are equivalent to
\(I\supseteq J\) and \(\sigma|_J=\tau\). Therefore
\[
\lambda^{(S)}_{J,\tau}
=
\sum_{I\supseteq J}\ \sum_{\sigma|_J=\tau} w_{I,\sigma}
=
\lambda^{(d)}_{J,\tau}.
\]
Thus the \(S\)-marginal is again a geometric witness copula at scale \(p_0\),
and the signed tail coefficients on the remaining coordinates are preserved.
\end{proof}

\section{Asymptotic analysis}
\label{sec:asymptotic}

Section~\ref{sec:margins} already established the key fixed-scale invariance
property of the canonical ray witness: once realized at scale \(p_0\), the same
complete signed tail family is preserved throughout the whole threshold range
\(0<p\le p_0\). The role of the present section is different, and it does not
introduce a new threshold-dependent theory of \(\lambda(p)\). We now record the
complete-case finite-threshold synthesis statement, identify the admissible
scale range, and formulate the vanishing-threshold viewpoint in which the same
\(p_0\)-free complete signed family is realized along families with
\(p_0\downarrow0\).

\subsection{Finite-threshold recovery on the ternary partition}

Fix $p\in(0,1/2)$, let $C$ be a copula on $[0,1]^d$, and let $U=(U_1,\dots,U_d)\sim C$. In this subsection, $p$ denotes an arbitrary finite threshold attached to the given copula $C$, whereas $p_0$ remains reserved for the geometric realization parameter of the witness construction. For each ternary state
\[
a\in\{L,M,U\}^d,
\]
let $B_a^{(p)}\subset[0,1]^d$ denote the corresponding ternary cell at threshold $p$, obtained by assigning in each coordinate the interval $L_p=[0,p]$, $M_p=[p,1-p]$, or $U_p=[1-p,1]$ according to the state $a$, with overlaps only at endpoints and hence as a ternary partition modulo null sets. Write
\[
q_a^{(p)}:=\Pr\!\bigl(U\in B_a^{(p)}\bigr),\qquad a\in\{L,M,U\}^d,
\]
for the ternary cell masses induced by the threshold $p$. For every nonempty $J\subseteq[d]$ and every $\tau\in\{L,U\}^J$, define the complete finite-threshold signed tail family by
\[
\lambda_{J,\tau}^{(p)}(C)
:=
\frac{1}{p}
\sum_{a\in T_{J,\tau}} q_a^{(p)}.
\]
Here \(\lambda^{(p)}(C)\) is an auxiliary finite-threshold readout of one fixed
copula \(C\) at one chosen threshold \(p\). It is not an additional modelling
object in this paper; the primary object remains the \(p_0\)-free complete signed
tail family \(\lambda\).

\begin{proposition}[Finite-threshold inversion and rescaling]
\label{prop:finite-p-moebius-recovery}
Let \(\lambda^{(p)}(C)=\bigl(\lambda_{J,\tau}^{(p)}(C)\bigr)\) be the complete
finite-threshold signed tail family induced by a copula \(C\) at scale \(p\), and let
\(w^{(p)}\) be the unique coefficient system recovered from
\(\lambda^{(p)}(C)\) by triangular inversion in
Proposition~\ref{prop:triangular-inversion}. Then for every
\((I,\sigma)\in\mathcal G_d\),
\[
w_{I,\sigma}^{(p)}
=
\frac{q_{a(I,\sigma)}^{(p)}}{p},
\]
where \(a(I,\sigma)\) is the unique noncentral ternary state with active set
\(I\) and sign pattern \(\sigma\).

Consequently,
\[
w_{I,\sigma}^{(p)}\ge 0
\qquad\text{for all }(I,\sigma)\in\mathcal G_d,
\]
the singleton normalization holds,
\[
\lambda_{\{i\},\LL}^{(p)}(C)=1,
\qquad
\lambda_{\{i\},\UU}^{(p)}(C)=1,
\qquad i=1,\dots,d,
\]
and
\[
p\sum_{(I,\sigma)\in\mathcal G_d} w_{I,\sigma}^{(p)}
=
\sum_{a\neq \MM^d} q_a^{(p)}
\le 1.
\]
\end{proposition}

\begin{proof}
Define
\[
\widetilde w_{I,\sigma}:=\frac{q_{a(I,\sigma)}^{(p)}}{p},
\qquad (I,\sigma)\in\mathcal G_d.
\]
By the definition of the support sets \(T_{J,\tau}\), we have
\[
\lambda_{J,\tau}^{(p)}(C)
=
\frac{1}{p}\sum_{a\in T_{J,\tau}} q_a^{(p)}
=
\sum_{I\supseteq J}
\sum_{\sigma:\,\sigma|_J=\tau}
\widetilde w_{I,\sigma}.
\]
Thus \(\lambda^{(p)}(C)=A\widetilde w\), where \(A\) is the same incidence map as
in Theorem~\ref{thm:w-to-lambda}. By the uniqueness statement in
Proposition~\ref{prop:triangular-inversion}, triangular inversion of the
complete system recovers exactly \(\widetilde w\). This gives the stated
identity
\(w_{I,\sigma}^{(p)}=q_{a(I,\sigma)}^{(p)}/p\),
that is, the fixed-threshold rescaling of the noncentral masses.

Nonnegativity follows because each \(q_a^{(p)}\) is a probability mass. The
singleton equalities follow from the exact one-dimensional margins of a copula.
Finally,
\[
p\sum_{(I,\sigma)\in\mathcal G_d} w_{I,\sigma}^{(p)}
=
\sum_{a\neq \MM^d} q_a^{(p)}
\le 1,
\]
because the noncentral ternary cells carry at most total mass \(1\).
\end{proof}

The next theorem characterizes exact synthesis precisely for complete
finite-threshold data.

\begin{theorem}[Complete Finite-Threshold Synthesis Theorem]
\label{thm:complete-finite-p-synthesis}
Let \(p\in(0,1/2)\), let \(\lambda=\bigl(\lambda_{J,\tau}\bigr)\) be a complete
signed tail family, and let \(w\) be the coefficient system recovered from
\(\lambda\) by triangular inversion. Then the following are equivalent:
\begin{enumerate}
\item \(\lambda\) is the complete finite-threshold signed tail family of some
copula at scale \(p\);
\item the recovered system satisfies
\[
w_{I,\sigma}\ge 0 \quad\text{for all }(I,\sigma)\in\mathcal G_d,
\qquad
\lambda_{\{i\},\LL}=\lambda_{\{i\},\UU}=1\ \text{for }i=1,\dots,d,
\]
and
\[
p\sum_{(I,\sigma)\in\mathcal G_d} w_{I,\sigma}\le 1;
\]
\item \(\lambda\) is realized by a geometric witness copula at scale \(p\).
\end{enumerate}
\end{theorem}

\begin{proof}
\((1)\Rightarrow(2)\) follows from
Proposition~\ref{prop:finite-p-moebius-recovery}.

\((2)\Rightarrow(3)\) is precisely the content of
Corollary~\ref{cor:complete-compatibility} together with the formal witness
construction of Proposition~\ref{prop:formal-witness-copula}.

\((3)\Rightarrow(1)\) is immediate, because a geometric witness copula is, by
construction, a copula realizing the same complete finite-threshold signed
family.
\end{proof}

For complete finite-threshold families, the witness construction
synthesizes exactly the complete data that any realizing copula induces on the
ternary partition.

\subsection{Finite-Threshold Admissibility}

With the fixed-threshold recovery in hand, we return to the \(p_0\)-free tail-level family \(\lambda\). At the tail
level, validity requires nonnegative recovered increments together with the
singleton normalization. Once these hold, the only remaining
finite-threshold constraint is geometric and is controlled entirely by the
residual central mass.

\begin{theorem}[Finite-Threshold Admissibility Theorem]
\label{thm:small-tail}
Let \(\lambda=\bigl(\lambda_{J,\tau}\bigr)\) be a complete signed tail family,
and let \(w\) be the unique coefficient system recovered from \(\lambda\) by
triangular inversion. Assume that
\[
w_{I,\sigma}\ge 0
\qquad\text{for all }(I,\sigma)\in\mathcal G_d,
\]
and that the singleton normalization holds,
\[
\lambda_{\{i\},\LL}=1,
\qquad
\lambda_{\{i\},\UU}=1,
\qquad i=1,\dots,d.
\]
Set
\[
S(w):=\sum_{(I,\sigma)\in\mathcal G_d} w_{I,\sigma},
\qquad
p_{\max}:=\min\!\left\{\tfrac12,\,\frac{1}{S(w)}\right\}.
\]
Then every \(p_0\in(0,p_{\max})\) admits an exact geometric witness realization
of the same complete signed tail family \(\lambda\). Moreover, if
\(1/S(w)<1/2\), the endpoint \(p_0=p_{\max}\) is also admissible. By contrast,
if \(p_0>p_{\max}\), no copula---witness or otherwise---can realize
\(\lambda\) as a complete finite-threshold signed tail family at scale \(p_0\).
\end{theorem}

This is an internal vanishing-threshold statement for the same complete signed tail
family \(\lambda\): the theorem concerns a family of finite-threshold copulas
carrying the same complete signed data, not a single limiting copula at
\(p_0=0\).

\begin{proof}
For any fixed \(p_0\in(0,1/2)\), Theorem~\ref{thm:complete-finite-p-synthesis}
shows that \(\lambda\) is realizable at scale \(p_0\) if and only if
\(p_0S(w)\le 1\). This is equivalent to \(p_0\le 1/S(w)\). Intersecting with the
intrinsic threshold range \((0,1/2)\) gives the stated admissible interval.
The endpoint statement is immediate when \(1/S(w)<1/2\), and non-realizability
for \(p_0>p_{\max}\) follows from the same equivalence.
\end{proof}

\begin{remark}[A dimension-only safe scale]
Under singleton normalization, every valid tail-level witness weight system
satisfies \(S(w):=\sum_{(I,\sigma)\in\mathcal G_d}w_{I,\sigma}\le 2d\).
Indeed, summing the lower- and upper-singleton normalization identities over
all coordinates gives
\(\sum_{(I,\sigma)\in\mathcal G_d}|I|\,w_{I,\sigma}=2d\). Since
\(|I|\ge 1\) for every noncentral generator, \(S(w)\le 2d\). The bound is
sharp, being attained by the purely singleton system
\(w_{\{i\},L}=w_{\{i\},U}=1\), \(i=1,\ldots,d\), with all higher-order
generator weights equal to zero. Consequently, for \(d\ge2\), every tail-level
valid complete signed family admits a finite-threshold witness realization at
every scale \(0<p_0\le 1/(2d)\). This is only a dimension-only safe scale; the
sharper family-specific admissible scale is \(p_0\le 1/S(w)\), subject to the
intrinsic constraint \(p_0<1/2\), as in
Theorem~\ref{thm:small-tail}.
\end{remark}

\begin{corollary}[Vanishing-Threshold Realization Corollary]
\label{cor:vanishing-threshold-family}
Under the assumptions of Theorem~\ref{thm:small-tail}, let
\((p_n)_{n\ge 1}\) be any sequence with \(p_n\downarrow 0\). Then, for all
sufficiently large \(n\), there exists a geometric witness copula \(C_{p_n}\)
such that the complete finite-threshold signed tail family induced by \(C_{p_n}\) at
scale \(p_n\) is exactly \(\lambda\). Moreover, for every such realization,
\[
q_{\MM^d}^{(p_n)} = 1-p_n S(w) \longrightarrow 1.
\]
In particular, the asymptotic object is the complete signed tail family \(\lambda\),
realized not by a new limiting copula at \(p=0\), but by a vanishing-threshold
family of finite-\(p\) copulas carrying the same complete signed system.
\end{corollary}

\begin{proof}
Since \(p_n\downarrow0\), we have \(p_n<p_{\max}\) for all sufficiently large
\(n\), so Theorem~\ref{thm:small-tail} yields the copulas \(C_{p_n}\). The
central-mass formula follows directly from
Proposition~\ref{prop:w-to-q}.
\end{proof}

\begin{remark}[Structural versus geometric constraints]
\label{rem:structural-vs-geometric}
For a complete signed tail family, the recovered Möbius increments are
precisely the recovered witness weights. In the complete finite-threshold
setting they coincide, by fixed-threshold rescaling, with the normalized
masses of the noncentral ternary cells. Thus
negative recovered increments and failure of the singleton normalization are
structural constraints: by
Proposition~\ref{prop:finite-p-moebius-recovery}, no copula can bypass them at
the given threshold. By contrast, a large value of \(S(w)=\sum w\) is not a
structural limitation on synthesis. It only determines the maximal admissible
scale
\[
p_{\max}=\min\!\left\{\tfrac12,\,\frac{1}{\sum w}\right\}.
\]
In this sense, the Möbius layer decides what is allowed in the tails, while
the finite-threshold geometry decides at which scales that same family can be
realized.
\end{remark}

\section{Discussion}
\label{sec:discussion}

After Section~\ref{sec:asymptotic}, the remaining questions are what has
been proved exactly, what lies outside the scope of the paper, and which next
questions are the most natural.

\subsection{What the framework adds}

The witness construction adds several distinct but closely related ingredients.
First, it provides a lean geometric language in which complete signed tail
families are linear and \(p_0\)-free. Second, in the complete case, triangular
inversion recovers the normalized noncentral ternary cell masses of any
realizing copula at a fixed finite threshold. In that setting, the residual
central mass is the only additional finite-threshold constraint.

It also separates complete and partial data cleanly. For complete signed tail families, the witness conditions are decided algebraically: triangular inversion recovers the relevant increments, and their signs and total mass determine tail-level validity and the admissible \(p_0\)-range.
For partial or noisy targets, the same witness parametrization
leads instead to polyhedral feasibility and approximation problems.

Appendix~\ref{app:poset-moebius} identifies the combinatorial backbone of the
witness map: it is the zeta transform of a nonzero signed ternary poset, and
the complete-case recovery is the corresponding M\"obius inversion.

The signed witness language also remains connected to earlier upper-tail
constructions: when restricted to upper-tail-only generators, the same
incidence algebra collapses to the binary complete upper-tail
incidence/inversion scheme underlying the complete unsigned multivariate
upper-tail regime.

\subsection{Scope, limits, and next steps}

Section~\ref{sec:margins} and Section~\ref{sec:asymptotic} together clarify the main structural point of the
paper. Section~\ref{sec:margins} proves that, in the canonical ray geometry, a witness copula
realized at scale $p_0$ preserves the same complete signed tail family
throughout the whole range $0<p\le p_0$, while Section~\ref{sec:asymptotic} identifies exactly
which scales $p_0$ are admissible for such a realization.

The paper does not claim uniqueness of the realizing copula for a given
complete family. It does not claim that partial or lower-order summaries
determine a unique completion without further assumptions. It does not claim
that every constructive copula model must reduce to the witness geometry,
only that every compatible complete signed tail family can be synthesized by a
witness copula. Nor does it identify the present vanishing-threshold reading
with the whole classical extreme-value copula theory. The point here is a
precise internal asymptotic scheme for the complete signed tail family
\(\lambda\).

Three next steps are particularly natural. One is to study the
higher-dimensional geometry of complete and partial compatibility regions more
systematically. Another is to develop statistical calibration and deterministic
model-selection rules inside the feasible witness class. A third is to compare
the expressive power of the witness framework with other constructive copula
families in concrete applied case studies.

Simple benchmark families already suggest that such comparisons may be
nontrivial. In any dimension \(d\ge2\), in the pure upper-tail restriction
of the witness framework, the specification \(\lambda^U_{\{i,j\}}=\beta\)
for all pairs \(\{i,j\}\subset[d]\), with \(\lambda^U_J=0\) for all
\(|J|\ge3\), is realizable exactly for \(0\le\beta\le 1/(d-1)\). Another
natural signed benchmark is the complete family with
\(\lambda_{\{i,j\},\tau}=\beta\) for all pairs \(\{i,j\}\subset[d]\) and all
\(\tau\in\{L,U\}^{\{i,j\}}\), together with
\(\lambda_{\{i\},L}=\lambda_{\{i\},U}=1\) for all \(i\in[d]\), and
\(\lambda_{J,\tau}=0\) for all \(|J|\ge3\). Within the witness framework this
is realizable exactly for \(0\le\beta\le 1/(2d-2)\), and at
\(\beta=1/(2d-2)\) all signed pairwise coefficients equal \(\beta\) although every higher-order coefficient vanishes.

\section{Conclusions}
\label{sec:conclusions}

We introduced a geometric witness construction for signed multivariate tail
dependence, indexed by active coordinate sets and sign patterns. For complete
signed tail families it yields a linear \(p_0\)-free parametrization in
witness weights together with explicit triangular inversion.

At a structural level, the witness construction is governed by the
zeta/M\"obius calculus of a nonzero signed ternary poset, with complete-case
recovery given by the corresponding M\"obius inversion; see
Appendix~\ref{app:poset-moebius}.

At a fixed threshold, the same inversion characterizes complete finite-threshold
compatibility through nonnegative recovered increments, singleton normalization,
and residual central mass. If the recovered increments are nonnegative, then
$p_{\max}=\min\{1/2,\,1/\sum w\}$ determines the admissible finite-scale range,
and every admissible $p_0$ yields an exact realization of the same complete
signed tail family. In the canonical ray geometry, this realized witness has a
further rigidity property: once realized at scale $p_0$, it preserves that same
family throughout the whole threshold range $0<p\le p_0$.

Computationally, the complete case is handled by direct inversion, whereas
partial, noisy, or inconsistent specifications lead to linear feasibility and
approximation problems in the same witness parametrization. This is particularly relevant in applications such as risk management and environmental modelling, where harmful deviations may occur in both directions and where reinforcing signed co-movements, in both bivariate and multivariate form, can materially change the overall risk profile. The same witness parametrization also provides an interpretable language for describing such signed extremal dependence patterns.

\section*{Acknowledgements}

The author used ChatGPT for language editing and technical assistance. All mathematical content and final judgments remain the author’s sole responsibility.

\section*{Reproducibility}
\label{sec:reproducibility}

The numerical results reported in this paper consist of direct inversion, linear-feasibility tests, linear
\(\ell^1\)-repair problems, and canonical simulation checks for the witness copula. Sections~\ref{sec:lp} and 
\ref{sec:simulation}
specify the linear models, benchmark inputs, simulation protocol, and validation logic, and Appendix E
contains an accompanying Python implementation, including the canonical witness sampler and the
variable-threshold diagnostic helpers used to compare theoretical and empirical values of
\(\lambda^{(p)}\). The numerical experiments were run in Python 3.13.5 using scipy 1.17.1, with
linear programs solved by the HiGHS backend via scipy.optimize.linprog.

\appendix
\numberwithin{figure}{section}

\section{Poset structure and M\"obius interpretation of the witness map}
\label{app:poset-moebius}

This appendix records the incidence-algebra structure underlying the witness
construction. It identifies a signed ternary poset naturally underlying the
witness indexing, shows that the witness map is its zeta transform, proves that
its intervals are Boolean, and interprets the complete-case recovery as the
corresponding M\"obius inversion.

\begin{definition}[Signed ternary poset]
Let
\[
E:=\{\MM,\LL,\UU\},
\]
equipped with the partial order
\[
\MM\prec \LL,\qquad \MM\prec \UU,
\]
while \(\LL\) and \(\UU\) are incomparable. Let
\[
P_d:=E^d
\]
with the product order, and let
\[
P_d^\times:=P_d\setminus\{(\MM,\dots,\MM)\}.
\]
\end{definition}

\begin{remark}
This is not the geometric order on the unit interval. It is the combinatorial
order of signed extension: the middle state \(\MM\) is the inactive state, while
\(\LL\) and \(\UU\) are the two possible active states.
\end{remark}

For \(x\in P_d\), write \(\supp(x):=\{i\in[d]:x_i\neq M\}\) and \(r(x):=\#\supp(x)\).

\begin{proposition}[Bijection with signed witness indices]
There is a canonical bijection between \(P_d^\times\) and
\[
G_d:=\{(I,\sigma):\emptyset\neq I\subseteq[d],\ \sigma\in\{\LL,\UU\}^I\}.
\]
Namely, for \(x\in P_d^\times\), let \(I(x):=\operatorname{supp}(x)\) and define
\(\sigma(x)\in\{\LL,\UU\}^{I(x)}\) by
\(
\sigma(x)_i=x_i,\; i\in I(x).
\)
Conversely, for \((I,\sigma)\in G_d\), define \(x(I,\sigma)\in P_d^\times\) by
\[
x(I,\sigma)_i=
\begin{cases}
\sigma_i,& i\in I,\\
\MM,& i\notin I.
\end{cases}
\]
Hence \(|P_d^\times|=|G_d|=3^d-1\).

The same indexing also identifies the complete signed tail family
\[
(\lambda_{J,\tau})_{\emptyset\neq J\subseteq[d],\ \tau\in\{\LL,\UU\}^J}
\]
with a function on \(P_d^\times\).
\end{proposition}

\begin{proof}
The two constructions are inverse to each other by inspection.
\end{proof}

\begin{proposition}[Order as signed extension]
\label{prop:order-as-signed-extension}
For \(x,y\in P_d^\times\),
\[
x\preceq y
\quad\Longleftrightarrow\quad
\forall i\in[d],\ x_i\neq \MM \Rightarrow x_i=y_i.
\]
Under the bijection above, this is exactly the signed extension relation
\[
J\subseteq I,
\quad
\sigma|_J=\tau,
\]
where \(x=x(J,\tau)\) and \(y=x(I,\sigma)\).
\end{proposition}

\begin{proof}
This is immediate from the definition of the product order on \(E^d\).
\end{proof}

For \(x\preceq y\), let \(D(x,y):=\{i\in[d]:x_i=M,\ y_i\neq M\}\), so \(|D(x,y)|=r(y)-r(x)\).

\begin{proposition}[Boolean intervals]
For any \(x\preceq y\) in \(P_d\), the interval
\(
[x,y]:=\{z\in P_d:x\preceq z\preceq y\}
\)
is canonically isomorphic to the Boolean lattice \(2^{D(x,y)}\).
\end{proposition}

\begin{proof}
If \(z\in[x,y]\), then for each \(i\in[d]\):
if \(x_i\neq \MM\), necessarily \(z_i=x_i=y_i\);
if \(x_i=\MM\) and \(y_i=\MM\), necessarily \(z_i=\MM\);
if \(x_i=\MM\) and \(y_i\neq \MM\), then \(z_i\in\{\MM,y_i\}\).
Thus \(z\) is determined uniquely by the subset
\[
S(z):=\{i\in D(x,y):z_i=y_i\}\subseteq D(x,y).
\]
Conversely, every subset \(S\subseteq D(x,y)\) determines a unique
\(z\in[x,y]\) by setting \(z_i=y_i\) for \(i\in S\), \(z_i=\MM\) for
\(i\in D(x,y)\setminus S\), and \(z_i=x_i\) otherwise. This identification
is order-preserving.
\end{proof}

\begin{corollary}[M\"obius function]
The M\"obius function of \(P_d\) is
\[
\mu(x,y)=
\begin{cases}
(-1)^{|D(x,y)|}=(-1)^{r(y)-r(x)},& x\preceq y,\\
0,& x\npreceq y.
\end{cases}
\]
The same formula holds on \(P_d^\times\).
\end{corollary}

\begin{proof}
Every interval \([x,y]\) in \(P_d\) is Boolean of dimension \(|D(x,y)|\),
so its M\"obius function is
\[
\mu(x,y)=(-1)^{|D(x,y)|}=(-1)^{r(y)-r(x)}
\]
for \(x\preceq y\), and \(0\) otherwise. The same formula holds on
\(P_d^\times\), since restricting from \(P_d\) to \(P_d^\times\) does not
change any interval between nonzero comparable elements.
\end{proof}

\begin{proposition}[Witness incidence matrix as zeta matrix]
Let \(\zeta\) be the zeta function of \(P_d^\times\),
\[
\zeta(x,y)=
\begin{cases}
1,& x\preceq y,\\
0,& \text{otherwise}.
\end{cases}
\]
Index the witness incidence matrix \(A\) by \(P_d^\times\). Then
\(
A_{x,y}=\zeta(x,y).
\)
Equivalently, if both \(w\) and \(\lambda\) are viewed as functions on
\(P_d^\times\), then
\[
\lambda(x)=\sum_{y\succeq x} w(y).
\]
\end{proposition}

\begin{proof}
By Proposition~\ref{prop:order-as-signed-extension}, \(x\preceq y\) is exactly the
signed extension condition \(J\subseteq I\) and \(\sigma|_J=\tau\), which is the
defining condition for the nonzero entries of the witness incidence matrix.
\end{proof}

\begin{corollary}[Witness weights as M\"obius coefficients]
Let \(\lambda\) be a complete signed tail family, viewed as a function on
\(P_d^\times\). Then its witness weights are the M\"obius transform of \(\lambda\):
\[
w(x)=\sum_{y\succeq x}\mu(x,y)\lambda(y),
\qquad x\in P_d^\times.
\]
Equivalently,
\[
w_{I,\sigma}
=
\sum_{K\supseteq I}\ \sum_{\rho|_I=\sigma}
(-1)^{|K|-|I|}\lambda_{K,\rho}.
\]
Thus the complete-case backward recovery is exactly M\"obius inversion on the
signed ternary poset.
\end{corollary}

\begin{proof}
The witness incidence matrix is the zeta matrix of \(P_d^\times\), so its inverse
is the M\"obius matrix.
\end{proof}

\section{Example incidence tables}
\label{app:incidence-tables}

This appendix collects larger explicit incidence tables that complement the
main discussion of the witness map \(\lambda = Aw\). The \(d=2\) case already
appears in Section~\ref{sec:tailmap}; here we record the larger signed example
in dimension \(d=3\) together with an upper-tail multivariate EHW incidence
matrix in dimension \(d=4\).

\subsection{Full signed incidence matrix in dimension \(d=3\)}

Order rows and columns by increasing active-set size,
\[
\begin{aligned}
& (1,L),\ (1,U),\ (2,L),\ (2,U),\ (3,L),\ (3,U), \ (12,LL),\ (12,LU),\ (12,UL),\ (12,UU),\\ &\ (13,LL),\ (13,LU),\ (13,UL),\ (13,UU),\ (23,LL),\ (23,LU),\ (23,UL),\ (23,UU),\ (123,LLL),\\ & \ (123,LLU),\ (123,LUL),\ (123,LUU),\ (123,ULL),\ (123,ULU),\ (123,UUL),\ (123,UUU).
\end{aligned}
\]
Then \(\lambda=A_3 w\), where \(A_3\in\{0,1\}^{26\times 26}\) is the signed
incidence matrix with entries
\[
(A_3)_{(J,\tau),(I,\sigma)}=\mathbf 1\{J\subseteq I,\ \sigma|_J=\tau\}.
\]
An explicit matrix form is:
\begingroup
\setlength{\arraycolsep}{4pt}
\renewcommand{\arraystretch}{0.8}
\tiny
\[
A_3 = \left( \begin{array}{*{26}{c}}
1 & 0 & 0 & 0 & 0 & 0 & 1 & 1 & 0 & 0 & 1 & 1 & 0 & 0 & 0 & 0 & 0 & 0 & 1 & 1 & 1 & 1 & 0 & 0 & 0 & 0 \\
0 & 1 & 0 & 0 & 0 & 0 & 0 & 0 & 1 & 1 & 0 & 0 & 1 & 1 & 0 & 0 & 0 & 0 & 0 & 0 & 0 & 0 & 1 & 1 & 1 & 1 \\
0 & 0 & 1 & 0 & 0 & 0 & 1 & 0 & 1 & 0 & 0 & 0 & 0 & 0 & 1 & 1 & 0 & 0 & 1 & 1 & 0 & 0 & 1 & 1 & 0 & 0 \\
0 & 0 & 0 & 1 & 0 & 0 & 0 & 1 & 0 & 1 & 0 & 0 & 0 & 0 & 0 & 0 & 1 & 1 & 0 & 0 & 1 & 1 & 0 & 0 & 1 & 1 \\
0 & 0 & 0 & 0 & 1 & 0 & 0 & 0 & 0 & 0 & 1 & 0 & 1 & 0 & 1 & 0 & 1 & 0 & 1 & 0 & 1 & 0 & 1 & 0 & 1 & 0 \\
0 & 0 & 0 & 0 & 0 & 1 & 0 & 0 & 0 & 0 & 0 & 1 & 0 & 1 & 0 & 1 & 0 & 1 & 0 & 1 & 0 & 1 & 0 & 1 & 0 & 1 \\
0 & 0 & 0 & 0 & 0 & 0 & 1 & 0 & 0 & 0 & 0 & 0 & 0 & 0 & 0 & 0 & 0 & 0 & 1 & 1 & 0 & 0 & 0 & 0 & 0 & 0 \\
0 & 0 & 0 & 0 & 0 & 0 & 0 & 1 & 0 & 0 & 0 & 0 & 0 & 0 & 0 & 0 & 0 & 0 & 0 & 0 & 1 & 1 & 0 & 0 & 0 & 0 \\
0 & 0 & 0 & 0 & 0 & 0 & 0 & 0 & 1 & 0 & 0 & 0 & 0 & 0 & 0 & 0 & 0 & 0 & 0 & 0 & 0 & 0 & 1 & 1 & 0 & 0 \\
0 & 0 & 0 & 0 & 0 & 0 & 0 & 0 & 0 & 1 & 0 & 0 & 0 & 0 & 0 & 0 & 0 & 0 & 0 & 0 & 0 & 0 & 0 & 0 & 1 & 1 \\
0 & 0 & 0 & 0 & 0 & 0 & 0 & 0 & 0 & 0 & 1 & 0 & 0 & 0 & 0 & 0 & 0 & 0 & 1 & 0 & 1 & 0 & 0 & 0 & 0 & 0 \\
0 & 0 & 0 & 0 & 0 & 0 & 0 & 0 & 0 & 0 & 0 & 1 & 0 & 0 & 0 & 0 & 0 & 0 & 0 & 1 & 0 & 1 & 0 & 0 & 0 & 0 \\
0 & 0 & 0 & 0 & 0 & 0 & 0 & 0 & 0 & 0 & 0 & 0 & 1 & 0 & 0 & 0 & 0 & 0 & 0 & 0 & 0 & 0 & 1 & 0 & 1 & 0 \\
0 & 0 & 0 & 0 & 0 & 0 & 0 & 0 & 0 & 0 & 0 & 0 & 0 & 1 & 0 & 0 & 0 & 0 & 0 & 0 & 0 & 0 & 0 & 1 & 0 & 1 \\
0 & 0 & 0 & 0 & 0 & 0 & 0 & 0 & 0 & 0 & 0 & 0 & 0 & 0 & 1 & 0 & 0 & 0 & 1 & 0 & 0 & 0 & 1 & 0 & 0 & 0 \\
0 & 0 & 0 & 0 & 0 & 0 & 0 & 0 & 0 & 0 & 0 & 0 & 0 & 0 & 0 & 1 & 0 & 0 & 0 & 1 & 0 & 0 & 0 & 1 & 0 & 0 \\
0 & 0 & 0 & 0 & 0 & 0 & 0 & 0 & 0 & 0 & 0 & 0 & 0 & 0 & 0 & 0 & 1 & 0 & 0 & 0 & 1 & 0 & 0 & 0 & 1 & 0 \\
0 & 0 & 0 & 0 & 0 & 0 & 0 & 0 & 0 & 0 & 0 & 0 & 0 & 0 & 0 & 0 & 0 & 1 & 0 & 0 & 0 & 1 & 0 & 0 & 0 & 1 \\
0 & 0 & 0 & 0 & 0 & 0 & 0 & 0 & 0 & 0 & 0 & 0 & 0 & 0 & 0 & 0 & 0 & 0 & 1 & 0 & 0 & 0 & 0 & 0 & 0 & 0 \\
0 & 0 & 0 & 0 & 0 & 0 & 0 & 0 & 0 & 0 & 0 & 0 & 0 & 0 & 0 & 0 & 0 & 0 & 0 & 1 & 0 & 0 & 0 & 0 & 0 & 0 \\
0 & 0 & 0 & 0 & 0 & 0 & 0 & 0 & 0 & 0 & 0 & 0 & 0 & 0 & 0 & 0 & 0 & 0 & 0 & 0 & 1 & 0 & 0 & 0 & 0 & 0 \\
0 & 0 & 0 & 0 & 0 & 0 & 0 & 0 & 0 & 0 & 0 & 0 & 0 & 0 & 0 & 0 & 0 & 0 & 0 & 0 & 0 & 1 & 0 & 0 & 0 & 0 \\
0 & 0 & 0 & 0 & 0 & 0 & 0 & 0 & 0 & 0 & 0 & 0 & 0 & 0 & 0 & 0 & 0 & 0 & 0 & 0 & 0 & 0 & 1 & 0 & 0 & 0 \\
0 & 0 & 0 & 0 & 0 & 0 & 0 & 0 & 0 & 0 & 0 & 0 & 0 & 0 & 0 & 0 & 0 & 0 & 0 & 0 & 0 & 0 & 0 & 1 & 0 & 0 \\
0 & 0 & 0 & 0 & 0 & 0 & 0 & 0 & 0 & 0 & 0 & 0 & 0 & 0 & 0 & 0 & 0 & 0 & 0 & 0 & 0 & 0 & 0 & 0 & 1 & 0 \\
0 & 0 & 0 & 0 & 0 & 0 & 0 & 0 & 0 & 0 & 0 & 0 & 0 & 0 & 0 & 0 & 0 & 0 & 0 & 0 & 0 & 0 & 0 & 0 & 0 & 1
\end{array} \right)
\]
\endgroup

\subsection{Upper-tail EHW multivariate incidence matrix in dimension \(d=4\)}
For the upper-tail-only case, order rows and columns by
\[
1,\ 2,\ 3,\ 4,\ 12,\ 13,\ 14,\ 23,\ 24,\ 34,\ 123,\ 124,\ 134,\ 234,\ 1234.
\]
Then
$\lambda^{U}=A^{U}_4\,w^{U}$ with entries $(A^{U}_4)_{J,I}=\mathbf 1\{J\subseteq I\}$,
and therefore
$w^{U}=(A^{U}_4)^{-1}\lambda^{U}$.
An explicit matrix form is:

\begingroup
\setlength{\arraycolsep}{4pt}
\renewcommand{\arraystretch}{0.8}
\tiny
\[
A^{U}_4 = \left( \begin{array}{*{15}{c}}
1 & 0 & 0 & 0 & 1 & 1 & 1 & 0 & 0 & 0 & 1 & 1 & 1 & 0 & 1 \\
0 & 1 & 0 & 0 & 1 & 0 & 0 & 1 & 1 & 0 & 1 & 1 & 0 & 1 & 1 \\
0 & 0 & 1 & 0 & 0 & 1 & 0 & 1 & 0 & 1 & 1 & 0 & 1 & 1 & 1 \\
0 & 0 & 0 & 1 & 0 & 0 & 1 & 0 & 1 & 1 & 0 & 1 & 1 & 1 & 1 \\
0 & 0 & 0 & 0 & 1 & 0 & 0 & 0 & 0 & 0 & 1 & 1 & 0 & 0 & 1 \\
0 & 0 & 0 & 0 & 0 & 1 & 0 & 0 & 0 & 0 & 1 & 0 & 1 & 0 & 1 \\
0 & 0 & 0 & 0 & 0 & 0 & 1 & 0 & 0 & 0 & 0 & 1 & 1 & 0 & 1 \\
0 & 0 & 0 & 0 & 0 & 0 & 0 & 1 & 0 & 0 & 1 & 0 & 0 & 1 & 1 \\
0 & 0 & 0 & 0 & 0 & 0 & 0 & 0 & 1 & 0 & 0 & 1 & 0 & 1 & 1 \\
0 & 0 & 0 & 0 & 0 & 0 & 0 & 0 & 0 & 1 & 0 & 0 & 1 & 1 & 1 \\
0 & 0 & 0 & 0 & 0 & 0 & 0 & 0 & 0 & 0 & 1 & 0 & 0 & 0 & 1 \\
0 & 0 & 0 & 0 & 0 & 0 & 0 & 0 & 0 & 0 & 0 & 1 & 0 & 0 & 1 \\
0 & 0 & 0 & 0 & 0 & 0 & 0 & 0 & 0 & 0 & 0 & 0 & 1 & 0 & 1 \\
0 & 0 & 0 & 0 & 0 & 0 & 0 & 0 & 0 & 0 & 0 & 0 & 0 & 1 & 1 \\
0 & 0 & 0 & 0 & 0 & 0 & 0 & 0 & 0 & 0 & 0 & 0 & 0 & 0 & 1
\end{array} \right)
\]
\endgroup

\section{Growth comparison with canonical-mixture counts}
\label{app:bell-comparison}

For comparison with the canonical-mixture language of~\cite{MB11}, let
\(B_d\) denote the \(d\)-th Bell number, that is, the number of partitions of
a \(d\)-element set into nonempty pairwise disjoint blocks. In the basic MB11
setting, \(B_d\) is also the number of canonical copulas when
countermonotonicity is not counted separately. It is also useful to compare
with the signed blockwise extension in which a common factor block may use
either \(f\) or its complement \(1-f\) on each active coordinate. Let
\(S_d\) denote the corresponding signed canonical count, here taken from OEIS
sequence A004211, \citep{OEISA004211}. Table~\ref{tab:3dminus1-vs-bell-signed} compares both
canonical-mixture counts with the witness count \(3^d-1\) of noncentral
ternary cells, equivalently of witness generators. From \(d\ge 9\) onward,
the witness count is already smaller than the Bell number, and relative to the
signed extension the gap is sharper by orders of magnitude. For the present
paper, smaller values of these ratios are favorable: they indicate that the
witness parametrization remains comparatively lean relative to both the
unsigned MB11 canonical family and its signed \(f\)/\(1-f\) extension as the
dimension grows.

\begin{table}[H]
\centering
\small
\caption{Comparison of the witness count \(3^d-1\) with the unsigned Bell
count \(B_d\) and the signed canonical-mixture comparator \(S_d\) from OEIS
A004211.}
\label{tab:3dminus1-vs-bell-signed}
\label{tab:APP-B-growth-comparison}
{\tiny
\begin{tabular}{r r r r r r}
\hline
\(d\) & \(3^d-1\) & \(B_d\) & \((3^d-1)/B_d\) & \(S_d\) & \((3^d-1)/S_d\) \\
\hline
2 & 8 & 2 & 4.000000 & 3 & 2.666667 \\
3 & 26 & 5 & 5.200000 & 11 & 2.363636 \\
4 & 80 & 15 & 5.333333 & 49 & 1.632653 \\
5 & 242 & 52 & 4.653846 & 257 & 0.941634 \\
6 & 728 & 203 & 3.586207 & 1539 & 0.473034 \\
7 & 2186 & 877 & 2.492588 & 10299 & 0.212254 \\
8 & 6560 & 4140 & 1.584541 & 75905 & 0.086424 \\
9 & 19682 & 21147 & 0.930723 & 609441 & 0.032295 \\
10 & 59048 & 115975 & 0.509144 & 5284451 & 0.011174 \\
11 & 177146 & 678570 & 0.261058 & 49134923 & 0.003605 \\
12 & 531440 & 4213597 & 0.126125 & 487026929 & 0.001091 \\
13 & 1594322 & 27644437 & 0.057672 & 5120905441 & 0.000311 \\
14 & 4782968 & 190899322 & 0.025055 & 56878092067 & 0.000084 \\
15 & 14348906 & 1382958545 & 0.010376 & 664920021819 & 0.000022 \\
16 & 43046720 & 10480142147 & 0.004107 & 8155340557697 & 0.000005 \\
17 & 129140162 & 82864869804 & 0.001558 & 104652541401025 & 0.000001 \\
18 & 387420488 & 682076806159 & 0.000568 & 1401572711758403 & 2.764184e-07 \\
19 & 1162261466 & 5832742205057 & 0.000199 & 19546873773314571 & 5.946022e-08 \\
20 & 3486784400 & 51724158235372 & 0.000067 & 283314887789276721 & 1.230710e-08 \\
\hline
\end{tabular}
}
\end{table}

\section{Hasse-type diagram for the signed witness index family in dimension $d=3$}
\label{app:hasse}

\begingroup
\setlength{\intextsep}{0.2\baselineskip}
\setlength{\abovecaptionskip}{3pt}
\setlength{\belowcaptionskip}{0pt}
\begin{figure}[htbp]
\centering
\includegraphics[width=1\textwidth,clip]{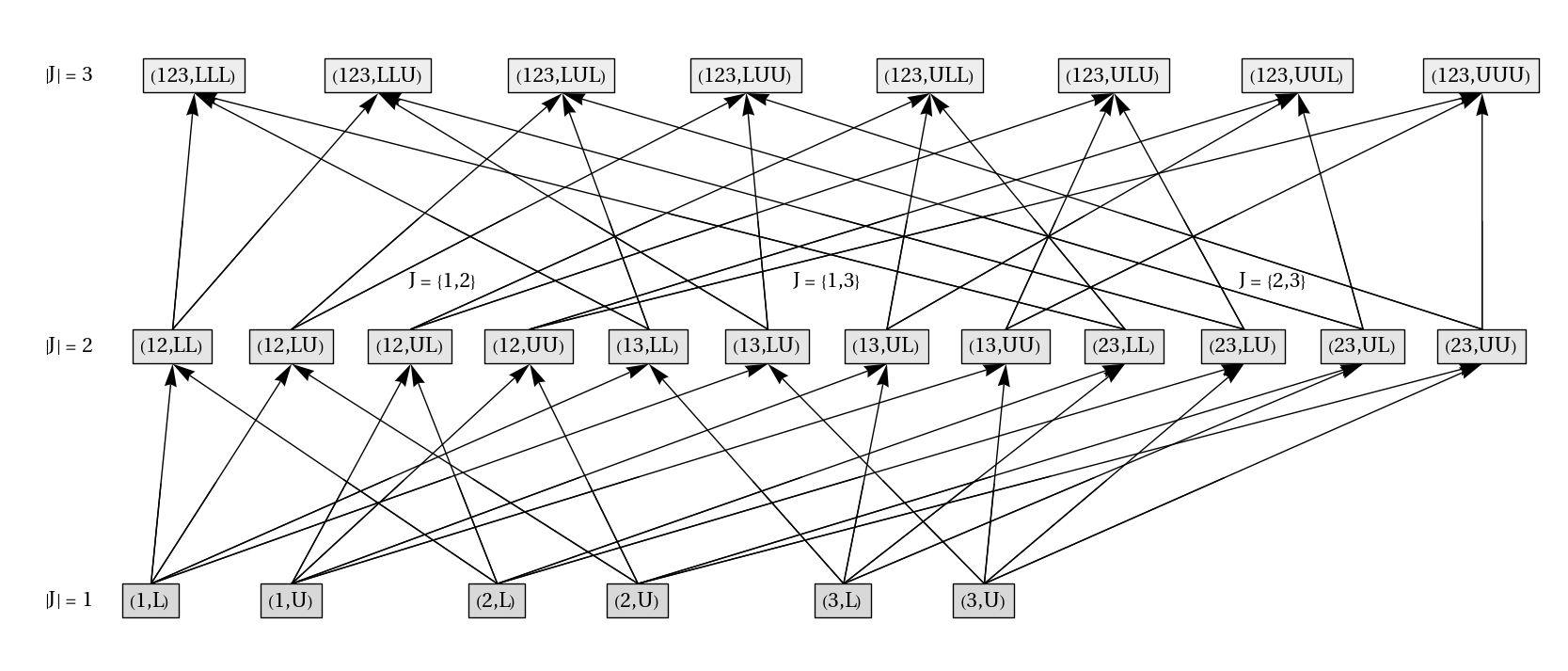}
\caption{Hasse-type diagram for the signed witness index family in dimension $d=3$. Nodes are indexed by signed active-set pairs $(J,\tau)$ and colored by active-set size $|J|$. Directed edges point to immediate supersets obtained by adding one signed coordinate. Thus the arrows encode the signed inclusion structure, while the recovery of witness weights in the complete case proceeds in the opposite direction, from the top layer to singleton indices by triangular back-substitution. The diagram makes the triangular back-substitution order underlying $\lambda = Aw$ explicit. Rendered with Wolfram Mathematica 14.3.}
\label{fig:hasse}
\end{figure}
\endgroup

\section{Python listing of the witness solver}
\label{app:python-witness-solver}

This appendix records the accompanying Python implementation used in the
computational parts of this paper.
The code is provided strictly as is, without any express or implied warranty,
without support or maintenance, and without any liability for errors,
omissions, failures, or consequences arising from its use. No commitment is
made to provide updates, corrections, or user support.

{\fontsize{5.17}{5.97}\selectfont
\setlength{\columnsep}{0.8em}
\begin{multicols}{2}
\begingroup
\setlength{\hfuzz}{100pt}
\begin{verbatim}
"""
Universal solver for geometric witness copulas.

This module supports:
    - EHW                (upper-tail pairwise specifications)
    - EHW-signed         (signed pairwise specifications)
    - EHW-MV             (upper-tail multivariate specifications)
    - EHW-signed-MV      (signed multivariate specifications)
    - complete, partial, noisy, or inconsistent target families

Core variables are witness weights w_{I,sigma}, indexed by active coordinate
sets I and sign patterns sigma on I.  The same code handles:
    - direct triangular inversion for complete specifications
    - LP feasibility for partial specifications
    - weighted L1 approximate recovery for noisy/infeasible targets
    - the bridge w <-> q to finite-p0 ternary (or upper-tail) cell masses

Conventions
-----------
1. Coordinates are 1-based: active sets are tuples like (1, 2, 5).
2. Sign alphabets are configured by `available_signs`:
       ('U',)         -> upper-tail-only frameworks (EHW, EHW-MV)
       ('L', 'U')     -> signed frameworks (EHW-signed, EHW-signed-MV)
3. A tail key is stored as:
       ((i1, ..., ik), (s1, ..., sk))
   where the coordinates are strictly increasing.
4. Exact margins are enforced through singleton coefficients:
       lambda_{ {i}, s } = 1   for each allowed sign s.

The module is intended as a robust Appendix-level implementation.
"""
from __future__ import annotations

from itertools import combinations, product
from typing import Dict, Iterable, List, Mapping, Optional, Sequence, Tuple

import numpy as np
from scipy.optimize import linprog

Sign = str
ActiveSet = Tuple[int, ...]
SignPattern = Tuple[str, ...]
TailKey = Tuple[ActiveSet, SignPattern]
State = Tuple[str, ...]


# ---------------------------------------------------------------------
# Basic key utilities
# ---------------------------------------------------------------------
def make_key(active: Sequence[int], signs: Sequence[str]) -> TailKey:
    """
    Canonicalize a tail key.

    Parameters
    ----------
    active:
        Active coordinate indices, in any order.
    signs:
        Sign pattern on the active coordinates.

    Returns
    -------
    TailKey
        Canonical key with sorted coordinates.
    """
    active = tuple(int(i) for i in active)
    signs = tuple(str(s) for s in signs)
    if len(active) == 0:
        raise ValueError("Active sets must be nonempty.")
    if len(active) != len(signs):
        raise ValueError("Active set and sign pattern must have the same length.")
    pairs = sorted(zip(active, signs), key=lambda x: x[0])
    active_sorted = tuple(i for i, _ in pairs)
    signs_sorted = tuple(s for _, s in pairs)
    if len(set(active_sorted)) != len(active_sorted):
        raise ValueError("Active coordinates must be distinct.")
    return active_sorted, signs_sorted


def format_key(key: TailKey) -> str:
    """Human-readable rendering of a tail key."""
    active, signs = key
    return f"{active}:{''.join(signs)}"


def validate_available_signs(available_signs: Sequence[str]) -> Tuple[str, ...]:
    signs = tuple(str(s) for s in available_signs)
    if len(signs) == 0:
        raise ValueError("At least one sign must be allowed.")
    if len(set(signs)) != len(signs):
        raise ValueError("available_signs must not contain duplicates.")
    if "M" in signs:
        raise ValueError("'M' is reserved for the finite-grid middle state.")
    return signs


def enumerate_keys(
    d: int,
    available_signs: Sequence[str] = ("L", "U"),
    orders: Optional[Iterable[int]] = None,
) -> List[TailKey]:
    """
    Enumerate all witness-generator / signed-tail keys for a given dimension.

    For |available_signs| = s, the total number of keys is:
        sum_{k=1}^d C(d, k) s^k.
    """
    if d < 2:
        raise ValueError("The solver is intended for d >= 2.")
    signs = validate_available_signs(available_signs)

    if orders is None:
        orders = range(1, d + 1)

    keys: List[TailKey] = []
    for k in orders:
        if not (1 <= int(k) <= d):
            raise ValueError("Requested orders must lie between 1 and d.")
        for active in combinations(range(1, d + 1), int(k)):
            for sigma in product(signs, repeat=int(k)):
                keys.append((active, sigma))
    return keys


def singleton_margin_targets(
    d: int,
    available_signs: Sequence[str] = ("L", "U"),
) -> Dict[TailKey, float]:
    """
    Exact-margin constraints represented as singleton tail coefficients.

    In the signed witness framework these are:
        lambda_{ {i}, s } = 1
    for every coordinate i and every allowed sign s.
    """
    signs = validate_available_signs(available_signs)
    return {((i,), (s,)): 1.0 for i in range(1, d + 1) for s in signs}


def complete_spec_template(
    d: int,
    available_signs: Sequence[str] = ("L", "U"),
    singleton_value: float = 1.0,
    default_other: float = 0.0,
) -> Dict[TailKey, float]:
    """
    Create a complete specification template.

    Singletons are set to singleton_value, all higher-order coefficients to
    default_other.
    """
    spec: Dict[TailKey, float] = {}
    for key in enumerate_keys(d, available_signs):
        active, _ = key
        spec[key] = float(singleton_value if len(active) == 1 else default_other)
    return spec


# ---------------------------------------------------------------------
# Linear maps
# ---------------------------------------------------------------------
def build_tail_matrix(
    d: int,
    target_keys: Sequence[TailKey],
    available_signs: Sequence[str] = ("L", "U"),
    generator_keys: Optional[Sequence[TailKey]] = None,
) -> Tuple[np.ndarray, List[TailKey]]:
    """
    Build the linear map A such that lambda = A w on the requested target keys.

    A[row, col] = 1 if the generator key (I, sigma) contributes to the target
    key (J, tau), i.e.
        J subseteq I
        sigma|_J = tau
    and 0 otherwise.
    """
    if generator_keys is None:
        generator_keys = enumerate_keys(d, available_signs)
    generator_keys = list(generator_keys)

    A = np.zeros((len(target_keys), len(generator_keys)), dtype=float)
    for r, target in enumerate(target_keys):
        J, tau = make_key(*target)
        tau_map = dict(zip(J, tau))
        J_set = set(J)

        for c, gen in enumerate(generator_keys):
            I, sigma = make_key(*gen)
            if not J_set.issubset(I):
                continue
            sigma_map = dict(zip(I, sigma))
            if all(sigma_map[j] == tau_map[j] for j in J):
                A[r, c] = 1.0
    return A, generator_keys


def tail_values_from_weights(
    d: int,
    weights: Mapping[TailKey, float],
    available_signs: Sequence[str] = ("L", "U"),
    target_keys: Optional[Sequence[TailKey]] = None,
) -> Dict[TailKey, float]:
    """
    Evaluate the linear map w -> lambda on the requested target keys.
    """
    if target_keys is None:
        target_keys = enumerate_keys(d, available_signs)
    target_keys = [make_key(*k) for k in target_keys]
    A, generator_keys = build_tail_matrix(
        d=d,
        target_keys=target_keys,
        available_signs=available_signs,
    )
    wvec = np.array([float(weights.get(g, 0.0)) for g in generator_keys], dtype=float)
    vals = A @ wvec
    return {k: float(v) for k, v in zip(target_keys, vals)}


def total_noncentral_mass(weights: Mapping[TailKey, float]) -> float:
    """Return sum w_{I,sigma}."""
    return float(sum(float(v) for v in weights.values()))


def admissibility_bound(weights: Mapping[TailKey, float]) -> float:
    """
    Return the largest admissible p0 for which the central mass remains
    nonnegative:
        p0 <= 1 / sum(w)
    """
    s = total_noncentral_mass(weights)
    return np.inf if s <= 0.0 else float(1.0 / s)


def check_exact_margins(
    d: int,
    weights: Mapping[TailKey, float],
    available_signs: Sequence[str] = ("L", "U"),
    tol: float = 1e-9,
) -> bool:
    """
    Check singleton margin equalities lambda_{ {i}, s } = 1.
    """
    lambdas = tail_values_from_weights(
        d=d,
        weights=weights,
        available_signs=available_signs,
        target_keys=list(singleton_margin_targets(d, available_signs).keys()),
    )
    return all(abs(v - 1.0) <= tol for v in lambdas.values())


# ---------------------------------------------------------------------
# Direct inversion in the complete case
# ---------------------------------------------------------------------
def direct_inversion_complete(
    d: int,
    lambda_complete: Mapping[TailKey, float],
    available_signs: Sequence[str] = ("L", "U"),
    tol: float = 1e-12,
) -> Dict[TailKey, float]:
    """
    Recover witness weights by triangular back-substitution.

    This requires a complete specification, i.e. values for all keys in the
    generator/tail family associated with the chosen sign alphabet.
    """
    generator_keys = enumerate_keys(d, available_signs)
    lambda_complete = {make_key(*k): float(v) for k, v in lambda_complete.items()}

    missing = [k for k in generator_keys if k not in lambda_complete]
    if missing:
        head = ", ".join(format_key(k) for k in missing[:5])
        raise ValueError(
            f"Complete specification is missing {len(missing)} coefficients. "
            f"First missing keys: {head}"
        )

    weights: Dict[TailKey, float] = {}
    keys_by_order = {
        k: [g for g in generator_keys if len(g[0]) == k]
        for k in range(1, d + 1)
    }

    for order in range(d, 0, -1):
        for key in keys_by_order[order]:
            I, sigma = key
            value = float(lambda_complete[key])
            correction = 0.0

            I_set = set(I)
            for super_key, super_weight in weights.items():
                K, rho = super_key
                if len(K) <= len(I):
                    continue
                if not I_set.issubset(K):
                    continue
                rho_map = dict(zip(K, rho))
                if all(rho_map[i] == s for i, s in zip(I, sigma)):
                    correction += float(super_weight)

            x = value - correction
            if abs(x) < tol:
                x = 0.0
            weights[key] = float(x)

    return weights


def complete_recovery_report(
    d: int,
    lambda_complete: Mapping[TailKey, float],
    available_signs: Sequence[str] = ("L", "U"),
    p0: Optional[float] = None,
    tol: float = 1e-9,
) -> Dict[str, object]:
    """
    Recover a complete witness solution and return diagnostics.
    """
    weights = direct_inversion_complete(
        d=d,
        lambda_complete=lambda_complete,
        available_signs=available_signs,
        tol=tol / 100.0,
    )

    recovered = tail_values_from_weights(
        d=d,
        weights=weights,
        available_signs=available_signs,
        target_keys=list(enumerate_keys(d, available_signs)),
    )
    residual = max(
        abs(recovered[k] - float(lambda_complete[make_key(*k)]))
        for k in recovered
    )

    nonnegative = all(v >= -tol for v in weights.values())
    margins_ok = check_exact_margins(d, weights, available_signs, tol=tol)

    report: Dict[str, object] = {
        "success": bool(nonnegative and margins_ok),
        "weights": weights,
        "lambda_recovered": recovered,
        "max_abs_lambda_residual": float(residual),
        "nonnegative": bool(nonnegative),
        "margins_ok": bool(margins_ok),
        "admissible_p0_max": admissibility_bound(weights),
    }

    if p0 is not None:
        q = build_q_from_weights(d, weights, p0, available_signs)
        report["q"] = q
        report["central_mass"] = float(q[tuple("M" for _ in range(d))])
        report["admissible_at_p0"] = bool(report["central_mass"] >= -tol)

    return report


# ---------------------------------------------------------------------
# LP formulations
# ---------------------------------------------------------------------
def _merge_exact_targets(
    d: int,
    target_lambdas: Mapping[TailKey, float],
    available_signs: Sequence[str],
    enforce_margins: bool,
    tol: float = 1e-12,
) -> Dict[TailKey, float]:
    """
    Merge user-specified coefficients with exact singleton margins.
    """
    merged: Dict[TailKey, float] = {}
    if enforce_margins:
        merged.update(singleton_margin_targets(d, available_signs))
    for k, v in target_lambdas.items():
        key = make_key(*k)
        val = float(v)
        if key in merged and abs(merged[key] - val) > tol:
            raise ValueError(
                f"Target {format_key(key)} conflicts with an exact margin constraint."
            )
        merged[key] = val
    return merged


def build_lp_model(
    d: int,
    target_lambdas: Mapping[TailKey, float],
    available_signs: Sequence[str] = ("L", "U"),
    enforce_margins: bool = True,
    p0: Optional[float] = None,
    mode: str = "feasibility",
    approx_weights: Optional[Mapping[TailKey, float]] = None,
    objective_costs: Optional[Mapping[TailKey, float]] = None,
) -> Dict[str, object]:
    """
    Build the LP model for one of the supported modes:
        - 'feasibility'
        - 'min_total_mass'
        - 'l1'

    In 'l1' mode the provided target_lambdas are treated as approximate targets,
    while exact margins remain exact (if enforce_margins=True).
    """
    signs = validate_available_signs(available_signs)
    generator_keys = enumerate_keys(d, signs)
    n_w = len(generator_keys)

    if mode not in {"feasibility", "min_total_mass", "l1"}:
        raise ValueError("mode must be 'feasibility', 'min_total_mass', or 'l1'.")

    if mode in {"feasibility", "min_total_mass"}:
        exact_targets = _merge_exact_targets(
            d=d,
            target_lambdas=target_lambdas,
            available_signs=signs,
            enforce_margins=enforce_margins,
        )
        eq_keys = list(exact_targets.keys())
        A_eq, _ = build_tail_matrix(
            d=d,
            target_keys=eq_keys,
            available_signs=signs,
            generator_keys=generator_keys,
        )
        b_eq = np.array([exact_targets[k] for k in eq_keys], dtype=float)

        c = np.zeros(n_w, dtype=float)
        if mode == "min_total_mass":
            if objective_costs is None:
                c[:] = 1.0
            else:
                c[:] = np.array(
                    [float(objective_costs.get(k, 0.0)) for k in generator_keys],
                    dtype=float,
                )

        A_ub = []
        b_ub = []
        if p0 is not None:
            if p0 <= 0:
                raise ValueError("p0 must be positive.")
            A_ub.append(np.ones(n_w, dtype=float))
            b_ub.append(float(1.0 / p0))

        bounds = [(0.0, None)] * n_w
        return {
            "mode": mode,
            "generator_keys": generator_keys,
            "A_eq": A_eq,
            "b_eq": b_eq,
            "A_ub": np.array(A_ub, dtype=float) if A_ub else None,
            "b_ub": np.array(b_ub, dtype=float) if b_ub else None,
            "c": c,
            "bounds": bounds,
            "n_w": n_w,
            "eq_keys": eq_keys,
            "available_signs": signs,
            "p0": p0,
        }

    # L1 approximate recovery
    exact_margin_targets = (
        singleton_margin_targets(d, signs) if enforce_margins else {}
    )
    margin_keys = list(exact_margin_targets.keys())
    approx_keys = [make_key(*k) for k in target_lambdas.keys()]

    m = len(approx_keys)
    total_vars = n_w + 2 * m
    slack_pos_offset = n_w
    slack_neg_offset = n_w + m

    rows: List[np.ndarray] = []
    rhs: List[float] = []

    if margin_keys:
        A_margin, _ = build_tail_matrix(
            d=d,
            target_keys=margin_keys,
            available_signs=signs,
            generator_keys=generator_keys,
        )
        for r, key in enumerate(margin_keys):
            row = np.zeros(total_vars, dtype=float)
            row[:n_w] = A_margin[r]
            rows.append(row)
            rhs.append(float(exact_margin_targets[key]))

    A_approx, _ = build_tail_matrix(
        d=d,
        target_keys=approx_keys,
        available_signs=signs,
        generator_keys=generator_keys,
    )
    for r, key in enumerate(approx_keys):
        row = np.zeros(total_vars, dtype=float)
        row[:n_w] = A_approx[r]
        row[slack_pos_offset + r] = -1.0
        row[slack_neg_offset + r] = 1.0
        rows.append(row)
        rhs.append(float(target_lambdas[key]))

    c = np.zeros(total_vars, dtype=float)
    if approx_weights is None:
        approx_weights = {k: 1.0 for k in approx_keys}
    for r, key in enumerate(approx_keys):
        wt = float(approx_weights.get(key, 1.0))
        c[slack_pos_offset + r] = wt
        c[slack_neg_offset + r] = wt

    A_ub = []
    b_ub = []
    if p0 is not None:
        if p0 <= 0:
            raise ValueError("p0 must be positive.")
        row = np.zeros(total_vars, dtype=float)
        row[:n_w] = 1.0
        A_ub.append(row)
        b_ub.append(float(1.0 / p0))

    bounds = [(0.0, None)] * total_vars
    return {
        "mode": mode,
        "generator_keys": generator_keys,
        "A_eq": np.array(rows, dtype=float) if rows else None,
        "b_eq": np.array(rhs, dtype=float) if rhs else None,
        "A_ub": np.array(A_ub, dtype=float) if A_ub else None,
        "b_ub": np.array(b_ub, dtype=float) if b_ub else None,
        "c": c,
        "bounds": bounds,
        "n_w": n_w,
        "approx_keys": approx_keys,
        "available_signs": signs,
        "p0": p0,
    }


def solve_witness_lp(
    d: int,
    target_lambdas: Mapping[TailKey, float],
    available_signs: Sequence[str] = ("L", "U"),
    enforce_margins: bool = True,
    p0: Optional[float] = None,
    mode: str = "feasibility",
    approx_weights: Optional[Mapping[TailKey, float]] = None,
    objective_costs: Optional[Mapping[TailKey, float]] = None,
    tol: float = 1e-9,
) -> Dict[str, object]:
    """
    Solve one of the LP problems supported by the witness framework.

    Returns
    -------
    dict
        Solution dictionary with weights, diagnostics, and (if p0 is given)
        the induced finite-grid masses q.
    """
    model = build_lp_model(
        d=d,
        target_lambdas=target_lambdas,
        available_signs=available_signs,
        enforce_margins=enforce_margins,
        p0=p0,
        mode=mode,
        approx_weights=approx_weights,
        objective_costs=objective_costs,
    )

    res = linprog(
        c=model["c"],
        A_eq=model["A_eq"],
        b_eq=model["b_eq"],
        A_ub=model["A_ub"],
        b_ub=model["b_ub"],
        bounds=model["bounds"],
        method="highs",
    )

    out: Dict[str, object] = {
        "success": bool(res.success),
        "message": res.message,
        "status": int(res.status),
        "mode": mode,
    }
    if not res.success:
        return out

    x = res.x
    generator_keys = model["generator_keys"]
    n_w = int(model["n_w"])

    all_weights = {k: float(x[i]) for i, k in enumerate(generator_keys)}
    sparse_weights = {k: v for k, v in all_weights.items() if abs(v) > tol}

    out["objective_value"] = float(res.fun)
    out["all_weights"] = all_weights
    out["weights"] = sparse_weights
    out["admissible_p0_max"] = admissibility_bound(all_weights)

    if p0 is not None:
        q = build_q_from_weights(d, all_weights, p0, available_signs)
        out["q"] = q
        out["central_mass"] = float(q[tuple("M" for _ in range(d))])

    # Reconstruct only the coefficients that were actually targeted.
    target_keys = [make_key(*k) for k in target_lambdas.keys()]
    out["lambda_recovered"] = tail_values_from_weights(
        d=d,
        weights=all_weights,
        available_signs=available_signs,
        target_keys=target_keys,
    )
    out["margins_ok"] = check_exact_margins(
        d=d,
        weights=all_weights,
        available_signs=available_signs,
        tol=tol,
    )

    if mode == "l1":
        approx_keys = model["approx_keys"]
        m = len(approx_keys)
        slack_pos_offset = n_w
        slack_neg_offset = n_w + m
        slack_plus = {
            approx_keys[i]: float(x[slack_pos_offset + i]) for i in range(m)
        }
        slack_minus = {
            approx_keys[i]: float(x[slack_neg_offset + i]) for i in range(m)
        }
        out["slack_plus"] = slack_plus
        out["slack_minus"] = slack_minus
        out["absolute_errors"] = {
            key: float(slack_plus[key] + slack_minus[key]) for key in approx_keys
        }

    return out


# ---------------------------------------------------------------------
# Finite-p0 bridge
# ---------------------------------------------------------------------
def build_q_from_weights(
    d: int,
    weights: Mapping[TailKey, float],
    p0: float,
    available_signs: Sequence[str] = ("L", "U"),
) -> Dict[State, float]:
    """
    Build finite-p0 cell masses q from witness weights.

    The state alphabet is:
        {M} union available_signs
    ordered naturally as L, M, U when present.
    """
    if p0 <= 0:
        raise ValueError("p0 must be positive.")
    signs = validate_available_signs(available_signs)

    alphabet = [s for s in ("L", "M", "U") if (s == "M" or s in signs)]
    q: Dict[State, float] = {}

    for state in product(alphabet, repeat=d):
        active = tuple(i + 1 for i, s in enumerate(state) if s != "M")
        if not active:
            continue
        sigma = tuple(s for s in state if s != "M")
        if any(s not in signs for s in sigma):
            continue
        q[state] = float(p0 * weights.get((active, sigma), 0.0))

    central = tuple("M" for _ in range(d))
    q[central] = float(1.0 - p0 * total_noncentral_mass(weights))
    return q


def build_weights_from_q(
    d: int,
    q: Mapping[State, float],
    p0: float,
    available_signs: Sequence[str] = ("L", "U"),
    tol: float = 1e-12,
) -> Dict[TailKey, float]:
    """
    Recover witness weights from noncentral finite-p0 cell masses.
    """
    if p0 <= 0:
        raise ValueError("p0 must be positive.")
    signs = validate_available_signs(available_signs)

    weights: Dict[TailKey, float] = {}
    for state, mass in q.items():
        if len(state) != d:
            raise ValueError("Each state in q must have length d.")
        state = tuple(state)
        active = tuple(i + 1 for i, s in enumerate(state) if s != "M")
        if not active:
            continue
        sigma = tuple(s for s in state if s != "M")
        if any(s not in signs for s in sigma):
            continue
        value = float(mass) / p0
        if abs(value) > tol:
            weights[(active, sigma)] = value
    return weights



def sample_canonical_witness(
    d: int,
    weights: Mapping[TailKey, float],
    p0: float,
    n_samples: int,
    rng: Optional[np.random.Generator] = None,
) -> np.ndarray:
    """
    Exact continuous sampler for the canonical witness copula.

    The sampler first draws the central component or one generator with
    probabilities p0*w_{I,sigma} and 1 - p0*sum w, then samples from the chosen
    canonical geometric component.
    """
    if not (0.0 < p0 < 0.5):
        raise ValueError("p0 must lie in (0, 1/2).")
    if n_samples < 1:
        raise ValueError("n_samples must be positive.")
    if rng is None:
        rng = np.random.default_rng()

    keys = list(weights.keys())
    probs = np.array([p0 * float(weights[k]) for k in keys], dtype=float)
    central = 1.0 - probs.sum()
    if central < -1e-12:
        raise ValueError("Negative central mass: witness system is not admissible at this p0.")
    probs = np.append(probs, max(0.0, central))
    probs /= probs.sum()

    sigma_maps = [dict(zip(I, sigma)) for I, sigma in keys]
    draws = rng.choice(len(probs), size=n_samples, p=probs)
    out = np.empty((n_samples, d), dtype=float)

    for r, j in enumerate(draws):
        if j == len(keys):
            out[r] = p0 + (1.0 - 2.0 * p0) * rng.random(d)
            continue

        z = float(rng.random())
        sigma_map = sigma_maps[j]
        row = np.empty(d, dtype=float)
        for i in range(1, d + 1):
            if i in sigma_map:
                row[i - 1] = p0 * z if sigma_map[i] == "L" else 1.0 - p0 * z
            else:
                row[i - 1] = p0 + (1.0 - 2.0 * p0) * rng.random()
        out[r] = row

    return out



def theoretical_lambda_p_canonical(
    d: int,
    weights: Mapping[TailKey, float],
    p: float,
    p0: float,
    available_signs: Sequence[str] = ("L", "U"),
    target_keys: Optional[Sequence[TailKey]] = None,
) -> Dict[TailKey, float]:
    """
    In the canonical ray geometry and for 0 < p <= p0, Section 4.5 implies
    lambda^(p) = lambda.
    """
    if not (0.0 < p <= p0):
        raise ValueError("This helper is intended only for 0 < p <= p0.")
    return tail_values_from_weights(
        d=d,
        weights=weights,
        available_signs=available_signs,
        target_keys=target_keys,
    )


def empirical_lambda_p(
    samples: np.ndarray,
    p: float,
    target_keys: Sequence[TailKey],
) -> Dict[TailKey, float]:
    """
    Empirical estimator
        hat lambda_{J,tau}(p) = (1 / (M p)) sum_m prod_{j in J} 1{U_j^(m) in T_{tau_j}(p)}.
    """
    if samples.ndim != 2:
        raise ValueError("samples must be a 2D array.")
    M, d = samples.shape
    if M < 1:
        raise ValueError("Need at least one sample.")
    if not (0.0 < p < 0.5):
        raise ValueError("p must lie in (0, 1/2).")

    out: Dict[TailKey, float] = {}
    for key in target_keys:
        J, tau = make_key(*key)
        mask = np.ones(M, dtype=bool)
        for j, s in zip(J, tau):
            x = samples[:, j - 1]
            if s == "L":
                mask &= (x <= p)
            elif s == "U":
                mask &= (x >= 1.0 - p)
            else:
                raise ValueError("Only L/U signs are allowed in target_keys.")
        out[(J, tau)] = float(mask.mean() / p)
    return out


def run_variable_p_diagnostics(
    d: int,
    weights: Mapping[TailKey, float],
    p0: float,
    p_values: Sequence[float],
    n_samples: int,
    available_signs: Sequence[str] = ("L", "U"),
    target_keys: Optional[Sequence[TailKey]] = None,
    rng: Optional[np.random.Generator] = None,
) -> Dict[str, object]:
    """
    Simulate one canonical witness sample and evaluate lambda^(p) versus
    empirical hat lambda^(p) on a grid of p values.
    """
    if rng is None:
        rng = np.random.default_rng()
    if target_keys is None:
        target_keys = list(enumerate_keys(d, available_signs))
    target_keys = [make_key(*k) for k in target_keys]

    samples = sample_canonical_witness(
        d=d,
        weights=weights,
        p0=p0,
        n_samples=n_samples,
        rng=rng,
    )
    lambda_target = tail_values_from_weights(
        d=d,
        weights=weights,
        available_signs=available_signs,
        target_keys=target_keys,
    )

    rows = []
    nonzero_keys = [k for k in target_keys if abs(lambda_target[k]) > 1e-14]
    zero_keys = [k for k in target_keys if abs(lambda_target[k]) <= 1e-14]

    for p in p_values:
        lam_theory = theoretical_lambda_p_canonical(
            d=d,
            weights=weights,
            p=p,
            p0=p0,
            available_signs=available_signs,
            target_keys=target_keys,
        )
        lam_emp = empirical_lambda_p(samples, p, target_keys)

        abs_err = {k: abs(lam_emp[k] - lam_theory[k]) for k in target_keys}
        rows.append({
            "p": float(p),
            "lambda_theory": lam_theory,
            "lambda_empirical": lam_emp,
            "max_abs_error": float(max(abs_err.values()) if abs_err else 0.0),
            "max_abs_error_nonzero": float(max((abs_err[k] for k in nonzero_keys), default=0.0)),
            "max_abs_leakage_zero": float(max((abs(lam_emp[k]) for k in zero_keys), default=0.0)),
        })

    return {
        "d": d,
        "p0": float(p0),
        "weights": {make_key(*k): float(v) for k, v in weights.items()},
        "lambda_target": lambda_target,
        "rows": rows,
    }

def sample_finite_grid_states(
    q: Mapping[State, float],
    n_samples: int,
    rng: Optional[np.random.Generator] = None,
) -> List[State]:
    """
    Sample discrete finite-grid states according to the cell masses q.

    This is a finite-grid diagnostic sampler.  It is intentionally distinct from
    a full geometric sampler on [0,1]^d.
    """
    if n_samples < 1:
        raise ValueError("n_samples must be positive.")
    if rng is None:
        rng = np.random.default_rng()

    states = list(q.keys())
    probs = np.array([float(q[s]) for s in states], dtype=float)
    if np.any(probs < -1e-12):
        raise ValueError("q contains negative masses.")
    probs = np.maximum(probs, 0.0)
    total = probs.sum()
    if total <= 0:
        raise ValueError("q must have positive total mass.")
    probs /= total

    draws = rng.choice(len(states), size=n_samples, p=probs)
    return [states[i] for i in draws]


# ---------------------------------------------------------------------
# Convenience builders for common model classes
# ---------------------------------------------------------------------
def make_ehw_pairwise_targets(matrix: Sequence[Sequence[float]]) \
 -> Tuple[int, Tuple[str, ...], Dict[TailKey, float]]:
    """
    Upper-tail pairwise EHW targets from a symmetric matrix with unit diagonal.
    """
    A = np.asarray(matrix, dtype=float)
    if A.ndim != 2 or A.shape[0] != A.shape[1]:
        raise ValueError("matrix must be square.")
    d = int(A.shape[0])
    if d < 2:
        raise ValueError("Need d >= 2.")
    if np.max(np.abs(np.diag(A) - 1.0)) > 1e-12:
        raise ValueError("The diagonal of the EHW matrix must be 1.")

    targets: Dict[TailKey, float] = {}
    for i in range(d):
        for j in range(i + 1, d):
            targets[((i + 1, j + 1), ("U", "U"))] = float(A[i, j])
    return d, ("U",), targets


def make_signed_pairwise_targets(
    d: int,
    pair_values: Mapping[TailKey, float],
) -> Tuple[int, Tuple[str, ...], Dict[TailKey, float]]:
    """
    Signed pairwise targets: only order-2 coefficients are supplied.
    """
    return d, ("L", "U"), {make_key(*k): float(v) for k, v in pair_values.items()}


def make_ehw_mv_targets(
    d: int,
    subset_values: Mapping[ActiveSet, float],
) -> Tuple[int, Tuple[str, ...], Dict[TailKey, float]]:
    """
    Upper-tail multivariate targets.  Each active subset I is interpreted as the
    all-U key (I, U...U).
    """
    targets: Dict[TailKey, float] = {}
    for active, value in subset_values.items():
        active = tuple(active)
        targets[make_key(active, ("U",) * len(active))] = float(value)
    return d, ("U",), targets


def make_signed_mv_targets(
    d: int,
    values: Mapping[TailKey, float],
) -> Tuple[int, Tuple[str, ...], Dict[TailKey, float]]:
    """
    General signed multivariate targets: arbitrary orders and arbitrary L/U sign
    patterns.
    """
    return d, ("L", "U"), {make_key(*k): float(v) for k, v in values.items()}


# ---------------------------------------------------------------------
# Five-dimensional signed benchmark from the paper
# ---------------------------------------------------------------------
def benchmark_signed_mv_5d(alpha: float) -> Tuple[int, Tuple[str, ...], Dict[TailKey, float]]:
    """
    Complete 5D signed benchmark specification from Section 9.
    """
    d = 5
    spec = complete_spec_template(
        d=d,
        available_signs=("L", "U"),
        singleton_value=1.0,
        default_other=0.0,
    )

    # Pair structure on {1,2} and {3,4}
    spec[make_key((1, 2), ("U", "L"))] = 1.0
    spec[make_key((1, 2), ("L", "U"))] = 1.0
    spec[make_key((3, 4), ("U", "L"))] = 1.0
    spec[make_key((3, 4), ("L", "U"))] = 1.0

    # Pairs involving coordinate 5
    for i in (1, 2, 3, 4):
        for signs in product(("L", "U"), repeat=2):
            spec[make_key((i, 5), signs)] = float(alpha)

    # Triple structure on {1,2,5} and {3,4,5}
    triples = [
        ((1, 2, 5), ("U", "L", "U")),
        ((1, 2, 5), ("U", "L", "L")),
        ((1, 2, 5), ("L", "U", "U")),
        ((1, 2, 5), ("L", "U", "L")),
        ((3, 4, 5), ("U", "L", "U")),
        ((3, 4, 5), ("U", "L", "L")),
        ((3, 4, 5), ("L", "U", "U")),
        ((3, 4, 5), ("L", "U", "L")),
    ]
    for key in triples:
        spec[make_key(*key)] = float(alpha)

    return d, ("L", "U"), spec


def benchmark_signed_mv_5d_expected_weights(alpha: float) -> Dict[TailKey, float]:
    """
    Explicit witness solution for the 5D benchmark.
    """
    weights: Dict[TailKey, float] = {}

    for key in [
        ((1, 2, 5), ("U", "L", "U")),
        ((1, 2, 5), ("U", "L", "L")),
        ((1, 2, 5), ("L", "U", "U")),
        ((1, 2, 5), ("L", "U", "L")),
        ((3, 4, 5), ("U", "L", "U")),
        ((3, 4, 5), ("U", "L", "L")),
        ((3, 4, 5), ("L", "U", "U")),
        ((3, 4, 5), ("L", "U", "L")),
    ]:
        weights[make_key(*key)] = float(alpha)

    for key in [
        ((1, 2), ("U", "L")),
        ((1, 2), ("L", "U")),
        ((3, 4), ("U", "L")),
        ((3, 4), ("L", "U")),
    ]:
        weights[make_key(*key)] = float(1.0 - 2.0 * alpha)

    weights[make_key((5,), ("U",))] = float(1.0 - 4.0 * alpha)
    weights[make_key((5,), ("L",))] = float(1.0 - 4.0 * alpha)
    return weights


def run_benchmark_5d_report(alpha: float, p0: Optional[float] = None, tol: float = 1e-9) \
 -> Dict[str, object]:
    """
    Solve the 5D benchmark by complete inversion and compare with the analytic
    witness solution.
    """
    d, signs, spec = benchmark_signed_mv_5d(alpha)
    report = complete_recovery_report(d, spec, signs, p0=p0, tol=tol)
    expected = benchmark_signed_mv_5d_expected_weights(alpha)

    all_keys = set(report["weights"].keys()) | set(expected.keys())
    max_diff = max(abs(report["weights"].get(k, 0.0) - expected.get(k, 0.0)) \
     for k in all_keys) if all_keys else 0.0

    report["expected_weights"] = expected
    report["max_abs_weight_diff_vs_closed_form"] = float(max_diff)
    report["alpha_compatible_interval_check"] = bool(0.0 <= alpha <= 0.25)
    return report


# ---------------------------------------------------------------------
# Minimal examples
# ---------------------------------------------------------------------
def minimal_demo_ehw() -> Dict[str, object]:
    """
    Small upper-tail EHW example (pairwise, d=3).
    """
    matrix = np.array([
        [1.0, 0.4, 0.3],
        [0.4, 1.0, 0.2],
        [0.3, 0.2, 1.0],
    ])
    d, signs, targets = make_ehw_pairwise_targets(matrix)
    return solve_witness_lp(
        d=d,
        target_lambdas=targets,
        available_signs=signs,
        enforce_margins=True,
        p0=0.10,
        mode="feasibility",
    )


def minimal_demo_signed_partial() -> Dict[str, object]:
    """
    Small signed partial example (d=3).
    """
    d = 3
    targets = {
        make_key((1, 2), ("U", "U")): 0.5,
        make_key((1, 3), ("U", "L")): 0.5,
        make_key((2, 3), ("L", "U")): 0.5,
    }
    return solve_witness_lp(
        d=d,
        target_lambdas=targets,
        available_signs=("L", "U"),
        enforce_margins=True,
        p0=0.10,
        mode="feasibility",
    )


def _print_sparse_weights(weights: Mapping[TailKey, float], tol: float = 1e-10) -> None:
    items = [(k, v) for k, v in weights.items() if abs(v) > tol]
    for key, value in sorted(items, key=lambda kv: (len(kv[0][0]), kv[0][0], kv[0][1])):
        print(f"  {format_key(key):>18s}  ->  {value:.12g}")


def main() -> None:
    print("=== minimal upper-tail EHW demo ===")
    demo = minimal_demo_ehw()
    print("success:", demo["success"])
    print("message:", demo["message"])
    if demo["success"]:
        print("nonzero weights:")
        _print_sparse_weights(demo["weights"])
        print("admissible p0 max:", demo["admissible_p0_max"])
    print()

    print("=== signed partial demo ===")
    demo2 = minimal_demo_signed_partial()
    print("success:", demo2["success"])
    print("message:", demo2["message"])
    if demo2["success"]:
        print("nonzero weights:")
        _print_sparse_weights(demo2["weights"])
        print("admissible p0 max:", demo2["admissible_p0_max"])
    print()

    print("=== 5D signed benchmark demo (alpha = 0.20) ===")
    bench = run_benchmark_5d_report(alpha=0.20, p0=0.10)
    print("success:", bench["success"])
    print("max abs lambda residual:", bench["max_abs_lambda_residual"])
    print("max abs weight diff vs closed form:", bench["max_abs_weight_diff_vs_closed_form"])
    if "central_mass" in bench:
        print("central mass at p0=0.10:", bench["central_mass"])
    print("nonzero weights:")
    _print_sparse_weights(bench["weights"])


if __name__ == "__main__":
    main()
\end{verbatim}
\endgroup
\end{multicols}
}


\begin{thebibliography}{99}

\bibitem[Boyd and Vandenberghe(2004)]{BoydVandenberghe2004}
S. Boyd and L. Vandenberghe.
\newblock \emph{Convex Optimization}.
\newblock Cambridge University Press, Cambridge, 2004.

\bibitem[De Luca and Rivieccio(2012)]{DeLucaRivieccio2012}
G.~De Luca and G.~Rivieccio.
\newblock Multivariate tail dependence coefficients for Archimedean copulae.
\newblock In A.~Di Ciaccio, M.~Coli, and J.~M. Angulo Ibáñez, editors,
\emph{Advanced Statistical Methods for the Analysis of Large Data-Sets},
pages 287--296. Springer, Berlin, 2012.

\bibitem[Embrechts et~al.(2016)Embrechts, Hofert, and Wang]{EHW2016}
P. Embrechts, M. Hofert, and R. Wang.
\newblock \emph{Bernoulli and Tail-Dependence Compatibility}.
\newblock The Annals of Applied Probability, 26(3):1636--1658, 2016.

\bibitem[Gudendorf and Segers(2010)]{GudendorfSegers2010}
G.~Gudendorf and J.~Segers.
\newblock Extreme-value copulas.
\newblock In P.~Jaworski, F.~Durante, W.~K. Härdle, and T.~Rychlik, editors,
\emph{Copula Theory and Its Applications}, pages 127--145. Springer, Berlin,
2010.

\bibitem[Jaynes(2003)]{Jaynes2003}
E.~T. Jaynes.
\newblock \emph{Probability Theory: The Logic of Science}.
\newblock Edited by G.~L. Bretthorst, Cambridge University Press, Cambridge, 2003.

\bibitem[Joe(2014)]{Joe2014}
H. Joe.
\newblock \emph{Dependence Modeling with Copulas}.
\newblock Chapman \& Hall/CRC, Boca Raton, 2014.

\bibitem[Krause et~al.(2018)Krause, Scherer, Schwinn, and Werner]{KrauseScherer2018}
D.~Krause, M.~Scherer, J.~Schwinn, and R.~Werner.
\newblock Membership testing for Bernoulli and tail-dependence matrices.
\newblock {\em Journal of Multivariate Analysis}, 168:240--260, 2018.

\bibitem[McNeil et~al.(2015)McNeil, Frey, and Embrechts]{MFE2015}
A.~J. McNeil, R. Frey, and P. Embrechts.
\newblock \emph{Quantitative Risk Management: Concepts, Techniques and Tools}.
\newblock Revised edition, Princeton University Press, Princeton, NJ, 2015.

\bibitem[Milek(2014)]{MB11}
J. Milek.
\newblock \emph{Multivariate B11 copula family for risk capital aggregation:
a copula engineering approach}.
\newblock Presentation slides for an invited talk in ``Talks in Financial and
Insurance Mathematics,'' ETH Zurich, March 6, 2014.
\newblock Prepared on invitation of Prof.\ Paul Embrechts.
\newblock Available at
\url{https://www2.math.ethz.ch/t3/fileadmin/math/ndb/00014/04970/ETH.March.6.2014.Zurich.Milek.v.3.pdf}.

\bibitem[Nelsen(2006)]{Nelsen2006}
R.~B. Nelsen.
\newblock \emph{An Introduction to Copulas}.
\newblock Springer, second edition, 2006.

\bibitem[OEIS Foundation Inc.(2026)]{OEISA004211}
OEIS Foundation Inc.
\newblock Integer sequence A004211.
\newblock \emph{The On-Line Encyclopedia of Integer Sequences}.
\newblock \url{https://oeis.org/A004211}.

\bibitem[Rényi(1961)]{Renyi1961}
A.~R\'enyi.
\newblock On measures of entropy and information.
\newblock In J.~Neyman, editor, \emph{Proceedings of the Fourth Berkeley
Symposium on Mathematical Statistics and Probability, Volume~1}, pages
547--561. University of California Press, Berkeley, 1961.

\end{thebibliography}
\end{document}